\documentclass[11pt]{article}
\usepackage{psfrag,
graphicx, 
euscript,amsfonts,amsmath,latexsym,amssymb,mathrsfs
}
\usepackage{fullpage}
\usepackage{subcaption}
\usepackage[toc,page]{appendix}
\usepackage[utf8]{inputenc}
\usepackage{tikz,inputenc,xcolor,amsthm,hyperref, color, url,amsfonts,amsmath,amssymb,multicol,graphicx,fancyhdr,tikz,pgfplots,comment,commath,bm}
\usepackage{pdflscape}
\usepgfplotslibrary{groupplots}

\usepackage[bottom=3cm, headsep=1cm,headheight=2cm]{geometry}

\DeclareMathOperator{\E}{\mathbb{E}}

\DeclareMathOperator{\R}{\mathbb{R}}

\renewcommand{\l}{\lambda}

\newcommand{\f}[2]{\frac{#1}{#2}}

\newtheorem{lemma}{Lemma}
\newtheorem{theorem}{Theorem}

\newtheorem{assumption}{Assumption}
\newtheorem{remark}{Remark}
\newtheorem{definition}{Definition}
\newtheorem{corollary}{Corollary}

\theoremstyle{definition}
\newtheorem{example}{Example}

\newcommand{\cL}{{\cal L}}

\newcommand{\cR}{{\cal R}}

\newcommand{\cX}{{\cal X}}

\newcommand{\bC}{\mathbb C}

\newcommand{\bE}{\mathbb E}

\newcommand{\bL}{{\mathbb L}}

\newcommand{\bN}{{\mathbb N}}

\newcommand{\bP}{{\mathbb P}}

\newcommand{\bR}{{\mathbb R}}

\newcommand{\bZ}{{\mathbb Z}}
\newcommand{\sG}{{\mathscr G}}
\newcommand{\sH}{{\mathscr H}}

\newcommand{\rd}{\mathrm{d}}

\renewcommand{\kappa}{\varkappa}

\begin{document}
\author{
A.~Goldenshluger
\\*[2mm]
{\small
Department of Statistics}
\\
{\small
University of Haifa
}
\\
{\small Haifa 31905, Israel}
\and 
D.T.~Koops
\\*[2mm]
{\small
Korteweg-de Vries institute}
\\
{\small
University of Amsterdam
}
\\
{\small Amsterdam 1098XH, The Netherlands}
}
\title{Nonparametric estimation of service time 
characteristics in infinite-server queues with nonstationary\\ Poisson input
\footnote{This research is supported by the Israel Science Foundation grant No. 361/15 and by the NWO Gravitation Project NETWORKS, Grant Number 024.002.003.}
%
}
\date{\today}
\maketitle
\begin{abstract} 
\noindent This paper provides a mathematical framework 
for estimation of the service time distribution and the expected service time of an infinite-server 
queueing system with a non--homogeneous 
Poisson arrival process, in the case of partial information, 
where only the number of busy servers are observed over time. 
The problem is reduced to  a statistical 
deconvolution problem, which is solved by 
using Laplace transform techniques and kernels for regularization. 
Upper bounds on the mean squared error 
of the proposed estimators are derived. 
Some concrete simulation experiments are performed 
to illustrate how the method can be applied and to provide some insight in the practical performance. 
\end{abstract}



\section{Introduction}
The advance of more and larger datasets leads to new questions 
in operations research and statistics. This paper can 
be placed in the intersection of these two fields. In particular, 
we study the statistical problems of estimating the distribution function and expectation 
of service times in an infinite-server queueing model 
in case of partial information. The information is incomplete 
in the sense that the number of busy servers is observed, but the individual customers cannot be tracked.

\paragraph{Infinite-server queue} First we provide some background on 
the $M/G/\infty$ queueing model, which is well studied and could be considered as a standard model. 
In such queues there are arrivals according to a 
homogeneous Poisson process, each customer is served independently of all other 
customers and customers do not have to queue for service, 
because there is an infinite number of servers. The model has a wide 
variety of applications in e.g.\ telecommunication, road traffic and hospital modeling. 
It can be used as an approximation to $M/G/n$ systems, where $n$ is relatively
large with respect to the arrival rate, but the model is also interesting 
in its own right. For example, it can be interpreted 
outside queueing theory as a model for the size of a population. 
In this paper we will use queueing terminology (servers, customers, etc.), 
but these terms could be adjusted according to the application. 
For example, `customers' could be cars travelling between two locations and their `service time' could be the travel time. 
In many applications it is seen that 
the arrival rate is not constant, 
but it varies over time. We will provide some examples below. 
This observation motivates studying the $M_t/G/\infty$ queue, 
where the arrival rate is assumed to be a \textit{nonhomogeneous} Poisson process. 
This model is still particularly tractable (cf. \cite{EMW1993}) and amenable for statistical analysis, as shown in this paper.

\paragraph{Statistical queueing problems} Queueing theory studies
probabilistic properties  of random processes in service systems   
on the basis of a complete specification of system parameters. In this paper we are interested in inverse problems when 
unknown characteristics of a system should be inferred from observations of the associated random processes. 
Typically such observations are incomplete in the sense that individual customers cannot be tracked as they go through the service 
system.
The importance of such  statistical inverse problems with incomplete observations 
was emphasized in \cite{Baccelli}. 
\par 
The service time distribution and its expected value are important performance metrics of the infinite-server queue (note that waiting times are identical to service times in infinite-server queues). Our goal is to estimate these characteristics of the $M_t/G/\infty$ system from observations of the queue--length process. 
Specifically, 
let $\{\tau_j, j\in \bZ\}$ be arrival epochs constituting 
a realization of the non--homogeneous  Poisson  point 
process 
on the positive real half line 
with intensity $\{\lambda(t), t\geq0\}$. The service times $\{\sigma_j, j\in \bZ\}$ are 
positive independent random variables with common distribution $G$, 
independent of $\{\tau_j, j\in\bZ\}$. Suppose that we observe the queue length (or: number of busy servers) $X(\cdot)$ 
restricted to the time interval $[0,T]$. The goal is to 
construct estimators of
the service time distribution~$G$ and 
the expected service time 
\[ 
\mu_G:=\int_0^\infty [1-G(t)]\rd t
\]  with provable accuracy guarantees.

In case of complete data, where 
it can be seen when 
each customer arrives and  leaves, 
the above 
estimation problems 
are trivial. The difficulties of the incomplete data problems lie 
in the fact that only macro level data is observed, i.e.\ only the number of busy servers are recorded. 

\paragraph{Applications} A non-exhaustive list of example applications of our analysis is given below:
\begin{itemize}
\item \textit{Traffic:} Toll plazas or sensors often count the number of cars entering and leaving a certain highway segment. Cars can overtake each other and therefore do not leave in the same order as they arrive. The problem is to determine the distribution of the speed of the cars over that segment. A solution that would make the problem trivial is to identify the car by recording its license plate. This comes with several downsides. For example, it could be costly to implement and there could be privacy concerns. There have been several attempts in the literature to solve this problem in another way, usually by trying to match vehicles at the upstream and downstream points based on non-unique signatures (for example, vehicle length), which can be obtained by a particular type of dual-loop detectors cf.\ e.g.\ \cite{Coifman}. Our approach applies to the harder case where single-loop detectors are used, i.e.\ where only the vehicle counts at the upstream and downstream points are known. 

\item \textit{Communication systems:} If two internet routers only track the 
timestamp of when a packet arrived, then what is the distribution of the packet 
flow duration between those routers? The importance of such statistical inverse problems was emphasized by \cite{Baccelli}.  In particular, the recent paper \cite{Antunes} proposes a sampling framework for measuring internet traffic based on the $M/G/\infty$ model. 
\item \textit{Biology: }As noted in \cite{Jansen2016}, the 
production of molecules may be `bursty', in the sense that periods of 
high production activity can be followed by periods of low activity. 
In \cite{Dob2009}, this is modeled using an interrupted Poisson process, 
i.e. a Poisson process that is modulated by a stochastic 
\textsc{on/off} switch. The number of active molecules can then 
be modeled as the number of jobs in a modulated infinite-server queue. By expectation 
maximization and maximum likelihood techniques it is possible to filter 
the most likely \textsc{on/off} arrival rate sample path, cf. \cite{Kris1997}. 
Subsequently, using such filtered paths, our methods can provide 
estimates of the lifetime distribution of molecules, when the molecules cannot be tracked separately but only the aggregate amount.
\end{itemize}

In all of these examples it would be unreasonable to assume that the arrival rate is 
constant. It is known that traffic is busier during rush hours, production of molecules is bursty, and internet activity is higher during the day than at night. Thus inhomogeneity of the Poisson process is an important feature in a variety of applications.

\paragraph{The paper contribution}
In this paper we deal with the statistical inversion problem of estimating the service time distribution $G$ and the expected service time $\mu_G$ in the $M_t/G/\infty$ queue, using only observations of the queue-length process. We approach the problem as a statistical deconvolution problem and develop nonparametric kernel estimators
based on Laplace transform techniques. Their accuracy is analyzed over suitable nonparametric classes of service time distributions.
In particular, we derive upper bounds on the worst--case root mean squared error (\textsc{rmse}) of the proposed estimators, and show how properties of the 
arrival rate function $\{\lambda(t), t\geq 0\}$ and service time distribution $G$ affect the estimation accuracy. Furthermore, we provide details on the implementation of the estimators. For example, we describe an adaptive estimation procedure for the distribution function and confirm its efficiency by a simulation study. 
\par 
Our results are based on a formula for the joint moment generating 
function of the queue--length process at different time instances. 
We provide a derivation of this result, which to the best of our knowledge 
has not appeared in the literature before and is of independent interest.

\paragraph{Current literature}

Our contribution is related to two different strands of research. First, a similar type of 
%
statistical inference questions in queueing 
theory was studied before, but then for the homogeneous case, i.e.\ the 
$M/G/\infty$ system; see, e.g., \ \cite{Brown}, \cite{Bingham}, \cite{Wichelhaus}, \cite{Goldenshluger}, \cite{Goldenshluger-b} and references
therein.
The analysis in case of the $M_t/G/\infty$ system  
is vastly different. This is due to the non-stationary nature of the 
$M_t/G/\infty$ queue, while the analysis for the $M/G/\infty$ relied on 
stationary measures. 
Second, 
similar deconvolution problems, such as 
density deconvolution,
have been studied in statistical literature; see, e.g.,   \cite{Zhang} and \cite{fan}. However, this strand of research 
typically considers models  with independent observations, and advocates the use of Fourier transform techniques.
In contrast, our setting is 
completely different:
the queue--length process in the $M_t/G/\infty$ model
is intrinsically dependent, and Fourier--based techniques are not 
applicable since the arrival rate 
$\{\lambda(t), t\geq 0\}$  need not be (square) integrable.
A connection can be drawn with  \cite{Bigot} where a deconvolution problem of estimating intensity of a non--homogeneous Poisson process on a finite interval was considered. We also mention recent work  \cite{Abramovich} where Laplace transform techniques were applied for signal deconvolution in a Gaussian model.

\paragraph{Outline}
The rest of the paper is organized as follows. 
In Section \ref{sec:preliminaries} we present the formula 
for the 
joint moment generating function of the queue--length process at 
different time instances, from which known results of \cite{EMW1993} can 
be derived. In particular, the covariance 
of the queue length at different points in time can be found, 
which plays an important role in the subsequent statistical analysis. 
In Section~\ref{sec:formulation} we  formulate the estimation problems and introduce necessary notation and assumptions.
Our main statistical results, Theorems~\ref{th:upper-bound}, \ref{th:expectation-1} and \ref{th:expectation-2},  
that provide  upper bounds on the 
risk of estimators of $G$ and $\mu_G$ are given in Sections~\ref{sec:estimation} and~\ref{sec:expectation}. 
Section~\ref{sec:experiments} discusses numerical implementation of the proposed estimators and 
presents simulation results.
%
Proofs of main results of this paper are given in Section~\ref{sec:proofs}.

\section{Properties of the queue--length process}
\label{sec:preliminaries}
In this section we derive some probabilistic results on properties of the queue--length process of the $M_t/G/\infty$ queue; 
these results provide 
a basis for construction of our estimators.
\par
Let $\{\tau_j, j\in \bZ\}$ be arrival epochs constituting 
a realization of a non--homogeneous  Poisson process  
on the real line  
with intensity $\{\lambda(t), t\in\bR\}$. The service times $\{\sigma_j, j\in \bZ\}$ are 
positive independent random variables with common distribution $G$, 
independent of $\{\tau_j, j\in\bZ\}$.
Assume that the system operates infinite time; then the queue--length
process \mbox{$\{X(t), t\in \bR\}$} is given by 
\begin{equation*}
 X(t) = \sum_{j\in \bZ} {\bf 1}(\tau_j \leq t, \sigma_j > t-\tau_j),\;\;\;t\in\bR.
\end{equation*}
\par
The next result provides formulas for the Laplace transform of  finite dimensional distributions
of $\{X(t), t\in \bR\}$.
\begin{theorem}\label{prop:1}
Let $\bar{G}(t):=1-G(t)$, 
\begin{equation}\label{eq:H(t)}
 H(t):= \int_0^\infty \bar{G}(u) \lambda (t-u)\rd u <\infty,\;\;\;\forall t\in \bR, 
\end{equation}
and  for $0\leq a<b\leq \infty$ let 
\[
H_{a,b}(t):=\int_a^b \bar{G}(u) \lambda (t-u)\rd u.
\]
Let $t_1 < \cdots < t_m$ be fixed points, and let $t_0=-\infty$
by convention; then
for any $\theta=(\theta_1,\ldots, \theta_m)$, $m\geq 1$ one has  
\begin{eqnarray}
&& \ln \bE_{G} \Big[\exp \Big\{\sum_{i=1}^m \theta_i X(t_i)\Big\}\Big] = 
\sum_{k=1}^m (e^{\theta_k}-1) H(t_k) 
\nonumber
\\
&&\;\;\;\;\;\;\;+\; \sum_{k=1}^{m-1} (e^{\theta_{k+1}}-1) \sum_{j=0}^{k-1} 
\Big(e^{\sum_{i=j+1}^k \theta_i}-1\Big)
H_{t_{k+1}-t_{j+1}, t_{k+1}-t_j}(t_{k+1}),
\label{eq:formula}
\end{eqnarray}
where $\bE_G$ stands for the expectation with respect to the probability measure $\bP_G$ generated by the 
queue--length process $\{X(t), t\in \bR\}$ when the service time distribution is $G$.
\end{theorem}
%
%
The following statement is an immediate consequence of the result of Theorem~\ref{prop:1} for specific cases 
$m=1$ and $m=2$.
\begin{corollary}\label{cor:1}
For any $t\in \bR$   
 \[ 
\ln \bE_{G} \big[\exp\{\theta X(t)\}\big] = (e^{\theta}-1) H(t),
\]
i.e.,   $X(t)$ is a Poisson random variable with parameter $H(t)$. Moreover, for any $t_2>t_1$
\begin{eqnarray*}
&& \ln \bE_{G} \big[\exp\big\{\theta_1 X(t_1)+\theta_2X(t_2)\big\}\big]
\\
&&\;\;=\; (e^{\theta_1}-1)H(t_1)+  (e^{\theta_2}-1)H(t_2)
\;+ (e^{\theta_2}-1)(e^{\theta_1}-1) H_{t_2-t_1,\infty}(t_2),
\end{eqnarray*}
and therefore 
\begin{equation}\label{eq:cov}
 {\rm cov}_{G} \big[X(t_1), X(t_2)\big] = H_{t_2-t_1,\infty}(t_2)=\int_{t_2-t_1}^\infty 
 \bar{G}(u) \lambda(t_2-u)
\rd u. 
\end{equation}
\end{corollary}
\par 
The results in Corollary~\ref{cor:1} are well known; they are proved, e.g.,
in \cite[Section~1]{EMW1993} using a completely different technique. 
We were not able, however, to locate  formula (\ref{eq:formula})
in the existing literature.

\section{Estimation problems, notation and assumptions}
\label{sec:formulation}
In this section we formulate estimation problems, 
introduce necessary notation and assumptions and present a general idea for the construction of estimators 
of linear functionals of  service time distribution $G$.
\subsection{Formulation of estimation problems}
Consider an $M_t/G/\infty$ queueing system with Poisson arrivals of 
intensity \mbox{$\{\lambda(t), t\geq 0\}$}.
Assume that $n$ independent realizations $X_1(t),\ldots, X_n(t)$ of 
the queue--length process $\{X(t), t\geq 0\}$ 
are observed on the interval $[0,T]$. 
Using the observations $\cX_{n,T}:=\{X_k(t), 0\leq t\leq T, k=1,\ldots,n\}$ we want to estimate 
the service time distribution $G$ and the expected service time $\mu_G:=\int_0^\infty [1-G(t)]\rd t$.
In the sequel, we will be interested in estimating the value $G(x_0)$ of $G$ at a single given point $x_0>0$.
\par 
It is worth noting that the observation scheme of this paper, where $n$ independent copies
of the queue--length process are given, is quite standard in statistics of non--stationary processes.
We refer, e.g., to 
\cite{kutoyants} where nonparametric estimation of  intensity 
of a non--homogeneous Poisson process was considered. On the other hand, the accuracy of the estimator also increases as a function of $\l_0$, where $\l_0$ scales the arrival rate. In other words, good estimations can be obtained even for $n=1$, as long as the arrival rate is high enough.
\par 
By estimators $\tilde{G}(x_0)=\tilde{G}(x_0; \cX_{n,T})$ and $\tilde{\mu}_G=\tilde{\mu}_G(\cX_{n,T})$ of  $G(x_0)$ and $\mu_G$ respectively
we mean  measurable functions of $\cX_{n,T}$. Their accuracy will be measured by the worst--case risk
over a set $\sG$ of distributions $G$. In particular, for a functional class $\sG$ the risk of $\tilde{G}(x_0)$ is  
defined by
\[
 \cR_{x_0}[\tilde{G}; \sG] := \sup_{G\in \sG} \Big[\bE_G |\tilde{G}(x_0) - G(x_0)|^2\Big]^{1/2},
\]
while the risk of $\tilde{\mu}_G$ is defined by
\[
 \cR[\tilde{\mu}_G; \sG] := \sup_{G\in \sG} \Big[\bE_G |\tilde{\mu}_G - \mu_G|^2\Big]^{1/2}.
\]
Our goal is to construct estimators of $G(x_0)$ and $\mu_G$ with provable accuracy guarantees over natural
classes $\sG$ of service time distributions.
\subsection{General idea for estimator construction}\label{subsec:strategy}
It follows from Theorem~\ref{prop:1} that expectation $H(t)=\bE_G[X(t)]$ 
of the queue--length process $X(t)$ is related to the service time 
distribution $G$ via the convolution equation (\ref{eq:H(t)}),
\[
H(t)=\bE_G[X(t)]=\int_0^\infty \bar{G}(u)\lambda(t-u)\rd u,\;\;\;\bar{G}(t)=1-G(t). 
\]
Therefore estimating  a linear functional $\theta(G)$ of $G$, such as $G(x_0)$ or $\mu_G$, from observation of $X(t)$ 
is a statistical  inverse problem of the deconvolution type.
\par
We will base construction of our estimators on the linear functional strategy that is frequently used
for solving ill--posed inverse problems. 
The main idea of this strategy is to find a pair of kernels, say, $K(x,t)$ and $L(x,t)$ such that the following 
two conditions are fulfilled:
\begin{itemize}
 \item[(i)] the integral $\int K(x,t)\bar{G}(t)\rd t$ approximates the value $\theta(G)$;
 \item[(ii)] the kernel $L$ is related to $K$ via the equation $\int K(x,t) \bar{G}(t)\rd t=\int L(x, t)H(t)\rd t$.
\end{itemize}
In view of Corollary~\ref{cor:1} and  
by condition (ii),  the statistic 
$\int L(x,t) X(t)\rd t$ 
is an unbiased estimator for the integral
$\int K(x,t)\bar{G}(t)\rd t$, which by (i)  approximates the value $\theta(G)$. Thus,   
$\int L(x,t) X(t)\rd t$ is 
a reasonable estimator of $\theta(G)$.

\subsection{Notation and assumptions}
In order to construct estimators in our settings  we will use the Laplace transform techniques. 
For this purpose we require the following notation.

\paragraph{Notation.}
For a generic locally integrable function $\psi$ on $\bR$  
the  bilateral Laplace transform 
is defined~by 
\[
 \widehat{\psi}(z):= \cL[\psi; z] = \int_{-\infty}^\infty e^{-zt} \psi(t) \rd t.
\]
The region of convergence of the integral on the right hand side is a vertical strip in the complex plane; it will be 
denoted by 
\[ \Sigma_\psi :=\big\{ z\in \bC: \sigma_\psi^- < {\rm Re}(z) < \sigma_{\psi}^+\big\}
\]
for some constants $-\infty \leq \sigma_{\psi}^-<\sigma_\psi^+\leq \infty$. If  
$\psi$ is supported on $[0,\infty)$ then $\widehat{\psi}(z):=\int_0^\infty  e^{-zt} \psi(t)\rd t$, and the  
corresponding region of convergence is a half--plane 
\[
\Sigma_\psi:=\{z\in \bC: {\rm Re}(z)>\sigma_\psi\},  \;\;\sigma_\psi\in \bR.
\]
The inversion formula for the Laplace transform is 
\[
 \psi(t) =  \frac{1}{2\pi i} \int_{s-i\infty}^{s+i\infty} \widehat{\psi}(z) e^{zt}\rd z= \frac{1}{2\pi}
 \int_{-\infty}^\infty \widehat{\psi} (s+i\omega) e^{(s+i\omega)t}\rd \omega,\;\;\;\sigma_\psi^-< s <\sigma_{\psi}^+.
\]

\paragraph{Assumptions on the arrival rate.}
As we will show in the sequel, estimation accuracy   
depends on the growth of $\lambda(t)$ 
at infinity and on the rate of decay of the Laplace transform
$\widehat{\lambda}(z)$ over vertical lines in the region of convergence $\Sigma_\lambda$.  
These properties  of the arrival rate  are 
quantified in the next two 
assumptions.

\begin{assumption} \label{as:lambda-1} 
The intensity function $\{\lambda(t), t\geq 0\}$ is a non--negative locally integrable function on $[0,\infty)$ with
the abscissa  of convergence   of the Laplace transform  $\sigma_\lambda\geq 0$.
\begin{itemize}
 \item[(a)]  If $\sigma_\lambda>0$ then there exists real number $\lambda_0>0$ such that 
  \begin{equation}\label{eq:lambda-growth}
 \lambda(t) \leq \lambda_0 e^{\sigma_\lambda t},\;\;t\geq 0.
 \end{equation}
 \item[(b)] If $\sigma_\lambda=0$ then there exist real numbers $\lambda_0> 0$, $p\geq 0$, and 
 $a_1\geq 0$, $a_2> 0$, $\max\{a_1, a_2\}=1$ 
 such that
 \begin{equation}\label{eq:lambda-growth-2}
  \lambda(t) \leq \lambda_0 (a_1 + a_2t^p),\;\;\;t\geq 0.
 \end{equation}
\end{itemize}
\end{assumption}
\begin{assumption} \label{as:lambda} 
The Laplace transform $\widehat{\lambda}(z)$ does not have zeros in $\{z\in \bC: {\rm Re}(z)>\sigma_\lambda\}$, and 
\begin{align}
|\widehat{\lambda}(\sigma +i\omega)|  \;\geq\;  \frac{\lambda_0 k_0}{[(\sigma-\sigma_\lambda)^2+ \omega^2]^{\gamma/2}},\;\;\;\;\;
\omega \in \bR,\;\;\sigma>\sigma_\lambda,
\label{eq:lambda-bounds-2}
\end{align}
where $\lambda_0\geq 0$, $\gamma\geq 0$ and  $k_0>0$.
\end{assumption}
Several remarks on these assumptions  are in order. 
\par 
Assumption~\ref{as:lambda-1} states  growth conditions on $\lambda$: (a) allows exponential growth, while (b) is a polynomial growth condition. The 
important case of bounded $\lambda$ is included in Assumption~\ref{as:lambda-1}(b) and corresponds to $p=0$. Here note that $a_2>0$ while $a_1\geq 0$, 
and a convenient normalization $\max\{a_1, a_2\}=1$ is used. 
The  growth conditions 
(\ref{eq:lambda-growth})
and (\ref{eq:lambda-growth-2}) guarantee 
that $\widehat{\lambda}$ is absolutely convergent in  
the half--plane $\Sigma_\lambda=\{z\in \bC: {\rm Re}(z)>\sigma_\lambda\}$ with 
$\sigma_\lambda>0$ and $\sigma_\lambda=0$ respectively. 
\par 
Assumption~\ref{as:lambda} states that $\widehat{\lambda}$ does not 
have zeros in $\Sigma_\lambda$. By the Riemann--Lebesgue lemma, $\widehat{\lambda}$ decreases along 
vertical lines in $\Sigma_\lambda$ (see, e.g., \cite[\S 23]{Doetsch}), and Assumption~\ref{as:lambda} 
stipulates 
the decrease rate.
Note that if $\gamma>1$ in (\ref{eq:lambda-bounds-2}) 
then $\int_{-\infty}^\infty |\widehat{\lambda}(\sigma+i\omega)|\rd \omega <\infty$ for all 
$\sigma>\sigma_\lambda$, and 
$\lambda(t)$ can be inverted from $\widehat{\lambda}$ by integrating over any vertical  line 
${\rm Re}(z)=\sigma>\sigma_\lambda$.    If $\gamma>1/2$ then 
$\int_{-\infty}^\infty |\widehat{\lambda}(\sigma+i\omega)|^2\rd\omega<\infty$ for any $\sigma>\sigma_\lambda$, 
and the Laplace inversion formula 
should be understood in the sense of the $\bL_2$--convergence (cf. \cite[Chapter~2, \S 10]{Widder46}).
\par\medskip
The next examples demonstrate that
Assumptions~\ref{as:lambda-1} and~\ref{as:lambda} hold in many cases of interest.
\begin{example}[Constant arrival rate]
Let $a\geq 0$ and $\lambda(t)=\lambda_0 {\bf 1}_{[a,\infty)}(t)$. It is obvious that 
Assumption~\ref{as:lambda-1} holds with $\sigma_\lambda=0$ and $p=0$. In addition, since
 \[
\widehat{\lambda}(\sigma+i\omega)=\frac{\lambda_0e^{a(\sigma+i\omega)}}{\sigma+i\omega}, \;\;\;\sigma> 0,
\]
and 
$|\widehat{\lambda}(\sigma+i\omega)| = \lambda_0 e^{a\sigma}(\sigma^2+\omega^2)^{-1/2}$,
Assumption~\ref{as:lambda}
holds with
$k_0=1$  and  $\gamma=1$.
Note that $\widehat{\lambda}(z)$ has a unique singularity point at $z=0$. 
\end{example}
\begin{example}[Polynomial arrival rate]\label{ex:2}
Let $\lambda(t)=\lambda_0 t^p {\bf 1}_{[0,\infty)}(t)$, $p\geq 0$; 
here Assumption~\ref{as:lambda-1} holds with $\sigma_\lambda=0$, $a_1=0$ and $a_2=1$.
Moreover,
$\widehat{\lambda}(\sigma+i\omega)=\lambda_0\Gamma(p+1)/(\sigma+i\omega)^{p+1}$, $\sigma> 0$; hence 
\[
|\widehat{\lambda}(\sigma+i\omega)|=\frac{\lambda_0\Gamma(p+1)}{ (\sigma^2+\omega^2)^{(p+1)/2}},\;\;\;\sigma>0.
\]
Thus  Assumption~\ref{as:lambda}
is valid with $\sigma_\lambda=0$, $\gamma=p+1$ and $k_0=\Gamma(p+1)$. The function $\widehat{\lambda}(z)$
has a unique singularity point at $z=0$.
\end{example}
\begin{example}[Sinusoidal arrival rate]
Let $\lambda(t)=\lambda_0[1+b\sin (t)]$ for $t\geq 0$, $0<b\leq 1$; then 
\[
 \widehat{\lambda}(\sigma+i\omega)= \frac{\lambda_0}{\sigma+i\omega} + \frac{\lambda_0 b}{\sigma^2+\omega^2+1},
 \;\;\;\sigma>0.
\]
Here 
\[
 |\widehat{\lambda}(\sigma+i\omega)| = \frac{\lambda_0}{\sigma^2+\omega^2}
 \bigg(\omega^2+ \frac{[\sigma (\sigma^2+\omega^2+1)+b(\sigma^2+\omega^2)]^2}{(\sigma^2+\omega^2+1)^2}\bigg)^{1/2}.
\]
It is evident that 
\[
\frac{\lambda_0}{\sqrt{\sigma^2+\omega^2}} \;\leq\;  |\widehat{\lambda}(\sigma+i\omega)|\;\leq\;
\frac{\lambda_0}{\sqrt{\sigma^2+\omega^2}}\Big(1+\frac{b}{\sqrt{\sigma^2+\omega^2}}\Big).
\]
Thus Assumption~\ref{as:lambda} holds with $\sigma_\lambda=0$, $\gamma=1$, $k_0=1$. Three 
singularity points of $\widehat{\lambda}$
are located on the convergence axis: $z=0$, $z=\pm i$.
\end{example}
\begin{example}[Exponential arrival rate] Let $\lambda(t)=\lambda_0 e^{\theta t}$, $t\geq 0$ for some $\theta>0$; then 
$\widehat{\lambda}(\sigma+i\omega)=\lambda_0 (\sigma+i\omega -\theta)^{-1}$, $\sigma>\sigma_\lambda:=\theta$. Here 
\[
 |\widehat{\lambda}(\sigma+i\omega)| = \frac{\lambda_0}{[(\sigma-\theta)^2+\omega^2]^{1/2}},\;\;\;\sigma>\theta.
\]
Thus Assumption~\ref{as:lambda} is fulfilled with $\sigma_\lambda=\theta$, $\gamma=1$ and $k_0=1$. 
\end{example}


\section{Estimation of the service time distribution}
\label{sec:estimation}

In this section we consider the problem of estimating the service time distribution.

\subsection{Estimator construction}

The estimator construction follows the linear functional strategy described in Section~\ref{subsec:strategy}.
\par
Let $K$ be a fixed bounded function supported on $[0,\infty)$. 
For real number $h>0$  
we define 
\begin{eqnarray}\label{eq:L}
 L_{h}(t):=\frac{1}{2\pi i} \int_{s-i\infty}^{s+i\infty} \frac{\widehat{K}(zh)}{\widehat{\lambda}(-z)} e^{zt} \rd z
 = \frac{e^{st}}{2\pi}\int_{-\infty}^\infty \frac{\widehat{K}((s+i\omega)h)}{\widehat{\lambda}(-s-i\omega)} e^{i\omega t}
 \rd \omega,\;\;\;\;\;t\in \bR.
\end{eqnarray}
Since the convergence region of $\widehat{\lambda}$ is 
 $\Sigma_\lambda=\{z\in \bC:{\rm Re}(z)> \sigma_\lambda\}$,  $\widehat{\lambda}(-z)$ is well defined in 
 $\{z\in \bC: {\rm Re}(z) < -\sigma_\lambda\}$. The kernel $K$ will be always chosen so that $\Sigma_K=\bC$, and 
 if Assumption~\ref{as:lambda} holds then 
 $\widehat{\lambda}(z)$ does not have zeros in $\Sigma_\lambda$. Then
 function $\widehat{K}(zh)/\widehat{\lambda}(-z)$ is analytic 
 in $\Sigma_L:=\{z\in \bC: {\rm Re}(z)<-\sigma_\lambda\}$, and   kernel $L_{h}$ does not depend 
 on real number $s$ provided that $s<-\sigma_\lambda$. 
\par 
The next result reveals a relationship  between kernels $L_{h}$ and $K_h(\cdot):=(1/h)K(\cdot/h)$ and provides motivation for construction
of our estimator.
%
\begin{lemma}\label{lem:L}
 Let $H(t)=\bE_{G} [X(t)]$ as defined in (\ref{eq:H(t)}), and 
 assume that 
 \begin{equation}\label{eq:conditions}
  \int_{-\infty}^{\infty} 
  \bigg|\frac{\widehat{K}((s+i\omega)h)}{\widehat{\lambda}(-s-i\omega)}\bigg| \rd \omega < 
  \infty,\quad
  \int_{0}^\infty |L_{h}(t-x_0)|\, H(t)\rd t <\infty,\;\;\forall x_0.
 \end{equation}
Then
 \[
  \int_{0}^\infty L_{h}(t-x_0) H(t) \rd t = 
  \int_{-\infty}^\infty \frac{1}{h} K\bigg(\frac{x-x_0}{h}\bigg)\bar{G}(x) \rd x,\;\;\forall x_0>0.
 \]
\end{lemma}
\begin{proof} 
The first condition in (\ref{eq:conditions}) ensures that $L_{h}(t)$ is well defined and finite for each $t$. In view of the second condition, 
by Fubini's theorem
\begin{equation*}
\int_{-\infty}^\infty L_{h}(t-x_0) H(t) \rd t
= \int_{0}^\infty \bar{G}(x) \int_{-\infty}^\infty 
 L_{h}(t-x_0)\lambda(t-x) \rd t \rd x. 
\end{equation*}
Now we show that the definition  (\ref{eq:L})
implies that function $L_{h}$ solves the equation
\begin{equation}\label{eq:equation-L}
 \int_{-\infty}^\infty 
 L_{h}(t-x_0)\lambda(t-x) \rd t = \frac{1}{h} K\bigg(\frac{x-x_0}{h}\bigg),\;\;\;\forall x.
\end{equation}
Indeed, the bilateral Laplace transform of the left-hand side is
\begin{align*}
 \int_{-\infty}^\infty e^{-zx}  \bigg[\int_{-\infty}^\infty 
 L_{h}(t-x_0)\lambda(t-x) \rd t \bigg] \rd x &= \widehat{\lambda}(-z) \int_{-\infty}^\infty e^{-zt} L_{h}(t-x_0)\rd t\\
 &= \widehat{\lambda}(-z) \widehat{L}_{h}(z)  e^{-zx_0}.
\end{align*}
On the other hand, 
\[
 \int_{-\infty}^\infty e^{-z x} \frac{1}{h}K\bigg(\frac{x-x_0}{h}\bigg)\rd x = e^{-zx_0} \widehat{K}(zh).
\]
It follows from the definition of $L_{h}(t)$ and the inversion formula for the bilateral Laplace transform
[cf. \cite[Chapter~VI, \S 5]{Widder46}]
that 
$\widehat{L}_{h}(z)= \widehat{K}(zh)/\widehat{\lambda}(-z)$
for all $z$ in a  region where $\widehat{K}(\cdot h)/\widehat{\lambda}(-\cdot)$  is analytic.
Then $L_{h}$ is indeed a 
solution to~(\ref{eq:equation-L}). 
\end{proof}
Now we are in a position to define the estimator of $G(x_0)$ based on  
$n$ independent realizations $X_1(t),\ldots, X_n(t)$ of the queue--length process $\{X(t), t\geq 0\}$ 
observed on the interval $[0,T]$. Let  
$\{t_j^{(k)}\}$ be an ordered set of the departure and arrival epochs of realization $X_k(t)$, $k=1,\ldots, n$, so that $X_k(\cdot)$ is constant in between any two sequential epochs. 
The estimator of $G(x_0)$ is defined as follows
\begin{align}
 \tilde{G}_{h}(x_0) &:= 1-\f1n \sum_{k=1}^n\int_0^T L_{h}(t-x_0) X_k(t)\dif t 
 \nonumber
 \\
 &= 1 - \f1n\sum_{k=1}^n\sum_j X_k(t_j^{(k)}) {\bf 1}_{[0,T]}\big(t_j^{(k)}\big) \int_{t_j^{(k)}}^{t_{j+1}^{(k)}\wedge T} 
 L_{h}(t-x_0) \dif t.
 \label{eq:est}
\end{align}
The construction depends on the design parameter $h$ that will be specified in the sequel.

\subsection{Upper bound on the risk}
Now we study the accuracy of the estimator 
$\tilde{G}_{h}(x_0)$ 
defined in (\ref{eq:est}). 
The risk of the estimator $\tilde{G}_{h}(x_0)$ will be analyzed under  local 
smoothness and moment assumptions on the probability distribution $G$.  In particular, we will  
assume that $G$ belongs to a local H\"older class and has bounded second moment. 
\begin{definition}
Let $A>0$, $\beta>0$ and $d>0$.
 We say that a function $f$ belongs to the class $\sH_{\beta, x_0}(A)$ if
 $f$ is $\lfloor \beta\rfloor := \max\{k\in \bN_0: k<\beta\}$ times continuously 
 differentiable on $(x_0-d, x_0+d)$ and 
 \[
  |f^{(\lfloor\beta\rfloor)}(x)-f^{(\lfloor\beta\rfloor)}(x^\prime)| \leq A |x-x^\prime|^{\beta-\lfloor\beta\rfloor},\;\;\;\forall x,x^\prime\in (x_0-d, x_0+d).
 \]
 \end{definition}
 \begin{definition}
  Let $M\geq 0$. We say that a distribution function $G$ supported on $[0,\infty)$ belongs to the class
  $\sG(M)$ if 
  \[
   \max_{k=1,2} \bigg\{ \int_0^\infty  kt^{k -1} [1-G(t)] \rd t\bigg\} \leq M.
  \]
   We denote also
 \[
 \bar{\sH}_{\beta, x_0}(A, M) := \sH_{\beta, x_0}(A)\;\cap\; \sG(M).
 \]
 \end{definition}
\par  
Let $K$ be a kernel supported on $[0,1]$ and  satisfying the following conditions: 
\begin{itemize}
 \item[(K1)] For a fixed positive integer $m$
 \begin{eqnarray*}
  \int_0^1 K(t) \rd t =1, \;\;\int_0^1 t^k K(t) \rd t=0,\;\;\;k=1,\ldots, m.
 \end{eqnarray*}
\item[(K2)] For a positive integer number $r$ kernel $K$ is $r$ times continuously differentiable 
on $\bR$ and 
\[
 \max_{x\in [0,1]}|K^{(j)}(x)| \leq C_K <\infty,\;\;\;\forall j=0,1,\ldots, r.
\]
\end{itemize}
\par
Note that conditions (K1)--(K2) are standard 
in nonparametric kernel estimation; 
see, e.g., \cite{Tsybakov}.

%

\begin{theorem}\label{th:upper-bound}
Suppose  that Assumptions~\ref{as:lambda-1}
and~\ref{as:lambda}
hold, and $G\in \bar{\sH}_{\beta, x_0}(A,M)$. Let 
$K$ be a kernel satisfying conditions (K1)--(K2)
with $m\geq \lfloor\beta\rfloor +1$ and $r> \gamma+1$.
Let $\tilde{G}_*(x_0)$ be the estimator (\ref{eq:est}) associated with kernel $K$ and parameter 
 \[
  h=h_*= \bigg(\frac{M \kappa}{A^2 \lambda_0 n}\bigg)^{1/(2\beta+2\gamma+1)},\;\;\;
  \kappa:=\left\{\begin{array}{ll}
\sigma_\lambda^{-1} e^{2\sigma_\lambda x_0}, & \sigma_\lambda>0,\\
a_1+a_2x_0^p, & \sigma_\lambda=0.
                \end{array}
\right.
 \]
If $T\to \infty$ so that $\frac{1}{T}\ln (\lambda_0 n)\to 0$ as $n\to\infty$
then 
\begin{equation}\label{eq:upper-bound}
 \limsup_{n\to\infty} 
 \Big\{\varphi_n^{-1} \cR_{x_0}[\tilde{G}_*; \bar{\sH}_{\beta,x_0}(A,M)]\Big\} \leq C_1,
\end{equation}
where $C_1$ may depend on $\beta$, $\gamma$ and $p$ only, and 
\[
 \varphi_n:=A^{(2\gamma+1)/(2\beta+2\gamma+1)} 
 \bigg(\frac{M\kappa}{\lambda_0 n}\bigg)^{\beta/(2\beta+2\gamma+1)}.
\]
\end{theorem}
\begin{remark}
 The risk of $\tilde{G}_*$ converges to zero at the nonparametric rate $O(n^{-\beta/(2\beta+2\gamma+1)})$ as $n\to\infty$. 
 The ``ill-posedness index'' $\gamma$ is determined by smoothness of function $e^{-\sigma t} \lambda(t) {\bf 1}_{[0,\infty)}(t)$ with appropriate 
 $\sigma$  on the entire real line.
 In particular, if $\lambda$ is continuously  differentiable on $(0,\infty)$ and $\lambda(0)>0$ then $\gamma=1$, and the resulting rate 
 is $O(n^{-\beta/(2\beta+3)})$. The deconvolution problem is much harder if $e^{-\sigma t} \lambda(t) {\bf 1}_{[0,\infty)}(t)$ is smooth on $\bR$: 
 for instance, under conditions
 of Example~\ref{ex:2} with $p>0$ the resulting rate is $O(n^{-\beta/(2\beta+2p+3)})$.
\end{remark}
\begin{remark}
Theorem~\ref{th:upper-bound} considers asymptotics as  $n\to\infty$.
Another natural asymptotic regime is the heavy traffic limit when the scale parameter $\lambda_0$ of the arrival intensity 
 tends to infinity while $n=1$. An inspection of the proof shows that the result of Theorem~\ref{th:upper-bound}
 remains true if asymptotics $n\to\infty$ with fixed $\lambda_0$ is replaced by $\lambda_0\to\infty$ with fixed $n$. 
\end{remark}
\begin{remark}
In general the rate of convergence $\varphi_n$ may depend on $x_0$, and this 
dependence is primarily determined by behavior of the arrival rate at infinity. 
In particular, if $\lambda(t)$ increases exponentially then the accuracy is proportional 
to $(e^{2\sigma_\lambda x_0})^{\beta/(2\beta+2\gamma+1)}$ and deteriorates rapidly with growth of $x_0$. 
Note however that if the $\lambda$ is bounded (here $\sigma_\lambda=0$ and $p=0$) then 
the upper bound does not depend on $x_0$. 
\end{remark}
\begin{remark}
In  cases $\sigma_\lambda>0$ and $\sigma_\lambda=0$, $p=0$
the statement of the theorem remains intact if
boundedness  of the second  moment of $G$ 
in the definition of $\bar{\sH}_{\beta,x_0}(A, M)$ is replaced by  the boundedness of the first moment. 
This is not so in the case $\sigma_\lambda=0$, $p>0$:
if boundedness of the first moment only
is assumed then  the dependence on $x_0$ is 
worse for large $x_0$ as the bound becomes proportional to  $x_0^{p+1}$.
\end{remark}

\section{Estimation of the expected service time}\label{sec:expectation}
In this section we consider the problem of estimating the expected service time 
$\mu_G=\int_0^\infty [1-G(t)] \rd t$ 
from
observations $\cX_{n,T}:=\{X_k(t), 0\leq t\leq T, k=1,\ldots,n\}$.

\subsection{Estimator construction} 

For real number $b>1/4$ (where $1/4$ is chosen arbitrarily) let $\psi_b (t)$ be an infinitely differentiable function on $\bR$ such that 
\[
\psi_b(t)=\left\{ \begin{array}{ll}
                   1, & t\in [0, b],\\
                   0, & t\notin [-1/4, b+1/4].
                  \end{array}\right.
\]
To define $\psi_b(t)$ on the intervals $[-1/4, 0]$ and $[b, b+1/4]$ we use standard construction. 
Let  $\psi_0(t)=e^{-\frac{1}{t(1-t)}}$, $t\in [0,1]$; then  on the interval $[-1/4, 0]$ 
where $\psi_b$ climbs from $0$ to $1$  we~put 
\[
 \psi_b(t) = 4c_0\int_{-1/4}^t \psi_0(4x+1)\rd x, \;\; -\tfrac{1}{4}\leq t\leq 0,
\]
while on the interval $[b, b+1/4]$ where $\psi_b(t)$ descends  from $1$ to $0$ we let 
\[
\psi_b(t) = 1- 4c_0\int_{b}^t \psi_0(4(x-b)) \rd x, \;\; b\leq t\leq b+\tfrac{1}{4}.
\]
Since $\psi_b$ is infinitely differentiable, the Laplace transform $\widehat{\psi}_{b}$  is an entire function.
\par 
Define 
\begin{equation}\label{eq:M(t)}
 J_b(t):=\frac{1}{2\pi} \int_{-\infty}^\infty 
 \frac{\widehat{\psi}_{b}(s+i\omega)}{\widehat{\lambda}(-s-i\omega)} e^{(s+i\omega) t} \rd t,\;\;\;s <-\sigma_\lambda,
\end{equation}
where $\sigma_\lambda$ is the abscissa of convergence of the Laplace transform $\widehat{\lambda}$ of $\lambda$.   
The following statement is a key result on  properties of the function $J_b(t)$.
\begin{lemma}\label{lem:2}
 If
 \[
 \int_{-\infty}^\infty \bigg|\frac{\widehat{\psi}_{b}(s+i\omega)}{\widehat{\lambda}(-s-i\omega)}\bigg| \rd t < \infty,\;\;
 \int_0^\infty |J_b(t)| H(t)\rd t <\infty
 \]
then 
\[
 \int_0^\infty J_b(t) H(t)\rd t = \int_0^\infty \psi_{b}(x) \bar{G}(x)\rd x.
\]
\end{lemma}
The proof of Lemma~\ref{lem:2} is omitted as it goes along the same lines as the proof of Lemma~\ref{lem:L}.
Note that the integral on the right hand side of the previous formula for large $b$ approximates $\mu_G$; this fact underlies 
the construction of our estimator.
\par 
The estimator of $\mu_G$ is defined as follows
\begin{equation}\label{eq:mu-hat}
 \tilde{\mu}_G = \frac{1}{n}\sum_{k=1}^n \int_0^T J_b(t) X_k(t) \rd t.
\end{equation}
The estimator depends on 
the tuning parameter $b$ which will be specified in what follows. 

\subsection{Upper bounds on the risk}
The next two statements establish upper bounds on the risk of $\tilde{\mu}_G$ under different growth conditions on 
the arrival rate, Assumptions~\ref{as:lambda-1}(a) and~\ref{as:lambda-1}(b) respectively.

\begin{theorem}\label{th:expectation-1}
 Let Assumptions~\ref{as:lambda-1}(a) and \ref{as:lambda} hold, and let $G\in \sG(M)$.
 Let $\tilde{\mu}_G$ be the estimator (\ref{eq:mu-hat})
  associated with parameter 
  \[
   b=b_*:=\frac{1}{2\sigma_\lambda} \bigg\{\ln \big(\sigma_\lambda^{-2\gamma+1}M\lambda_0 n \big)- 3
 \ln \Big[\frac{1}{2\sigma_\lambda}\ln \big(\sigma_\lambda^{-2\gamma+1}M\lambda_0n)\Big]\bigg\}-\frac{1}{4}.
  \]
If $T\to \infty$ and $\frac{1}{T}[\ln (\lambda_0n)]^2 \to 0$ as $n\to\infty$ then 
\[
 \limsup_{n\to\infty} \Big\{\varphi_n^{-1} \cR[\tilde{\mu}_G; \sG(M)]\Big\}\leq C_1,\;\;\;
 \varphi_n := M \sigma_\lambda  \ln \bigg(\frac{M\lambda_0 n}{\sigma_\lambda^{2\gamma-1}}\bigg),
\]
and $C_1$ is a positive constant depending on $\gamma$ and  $k_0$.
\end{theorem}

\begin{theorem}\label{th:expectation-2}
Let Assumptions~\ref{as:lambda-1}(b) and \ref{as:lambda} hold, and let $G\in \sG(M)$.
Let $\tilde{\mu}_G$ be the estimator (\ref{eq:mu-hat})
  associated with parameter 
\[
 b=b_*:=(M\lambda_0n)^{1/(p+2)}.
 \]
If $\frac{1}{T}(\lambda_0 n)^{1/(p+2)} \ln (\lambda_0 n)\to 0$ as $n\to \infty$ then 
\[
 \limsup_{n\to\infty} \Big\{\varphi_n^{-1}\cR[\tilde{\mu}_G; \sG(M)] \Big\} \leq C_2, \;\;\;
 \varphi_n:=M^{(p+1)/(p+2)} (\lambda_0n)^{-1/(p+2)},
\]
where $C_2$ is a positive  constant depending on $\gamma$, $k_0$ and $p$ only.
\end{theorem}
\begin{remark}
 The results of Theorems~\ref{th:expectation-1} and~\ref{th:expectation-2} hold without any conditions
 on smoothness of $G$.  In particular, $G$ may be a discrete distribution with bounded second moment.
\end{remark}

\begin{remark}
 Theorems~\ref{th:expectation-1} and~\ref{th:expectation-2} demonstrate that accuracy of $\tilde{\mu}_G$ is primarily determined
 by the growth of the arrival rate at infinity. In particular, if the arrival rate  increases exponentially, i.e. $\sigma_\lambda>0$,
 then the risk of 
 $\tilde{\mu}_G$ converges to zero at a slow logarithmic rate. At the same time, if the arrival rate is bounded, i.e. 
 $\sigma_\lambda=0$ and $p=0$, then 
 the risk tends to zero at the parametric rate. A close inspection of the proof shows that growth of the arrival rate manifests itself
 in the growth of the variance of $X(t)$ which, in its turn, affects the rate of convergence. 
\end{remark}

\section{Implementation and numerical examples}
\label{sec:experiments}

In this section we provide details on implementation of the proposed estimators. 
In particular, we discuss numerical methods for calculating 
kernels $L_h$ and $J_b$ and an adaptive scheme  for bandwidth selection. 
We also conduct a small simulation study in order to illustrate practical behavior of the proposed estimators. 

\subsection{Implementation issues}
%
In our simulations we consider a Gaussian kernel, i.e.
\[
K(t) = \f{1}{\sqrt{2\pi}} e^{-t^2/2}, \quad\text{and thus}\quad \widehat{K}(z)=e^{z^2/2},
\]
for all $z\in\mathbb{C}$. 
The reason we use this kernel is that in some 
cases $L_h$ can be computed explicitly. 
If such analytical expressions are not available, we resort 
to numerical integration and inversion techniques. For consistency, 
we still use the Gaussian kernel in the numerical cases, even though this is obviously not necessary.
\par
It is computationally somewhat expensive to implement the theoretical kernel 
$J_b$ in (\ref{eq:M(t)}) as covered in Theorems \ref{th:expectation-1} and \ref{th:expectation-2} because 
the integral of $\psi_0$ does not have an explicit form, and neither does its transform. 
From that point of view it is more attractive to work with a slightly different kernel $J_b$, 
in particular
\[
\widehat{J}_b(z) := \f{\widehat{\psi}(z)}{\widehat{\lambda}(-z)} = 
\f{(1-e^{-zb}) e^{z^2h^2/2}}{z \widehat{\lambda}(-z)},
\]
%
i.e.\ $\psi$ is a convolution of the 
indicator function ${\bf 1}_{[0,b]}(\cdot)$ and $(1/h)K(\cdot/h)$, 
where $K$ is a standard Gaussian kernel. 
In some cases we let $\psi$ to be a convolution of the indicator function ${\bf 1}_{[0, b+x_1]}(\cdot)$ and 
$(1/h)K(\cdot-x_1/h)$
for some $x_1<0$. The shift to the left with $-x_1$ is there in order to get rid of some bias
caused by the boundary effects as $h$ approaches  zero. 
\subsubsection*{Exact calculation of kernels $L_h$ and $J_b$} 
Before we consider cases where only numerical calculation of kernels
is possible, we present the following examples where explicit expressions for kernels $L_h$ and $J_b$ 
are available. 
\begin{example}[Constant arrival rate]
\label{ex:const}
Suppose that $\l(t)=\l_0 {\bf 1}_{[0,\infty)}(t)$; then  
$\widehat\l(-z)=-\l_0/z$, for ${\rm Re}(z)<0$. 
It can be checked that the assumptions in Lemma \ref{lem:L} are satisfied, hence $L_h$ is well defined, for $c<0$, by
\begin{align*}
L_h(t) &= -\f{1}{2\pi\l_0} \int_{-\infty}^{\infty} 
(c+i\omega) e^{(c+i\omega)t+\f12h^2(c+i\omega)^2}\dif\omega=\frac{t e^{-\frac{t^2}{2 h^2}}}{\sqrt{2 \pi }\l_0 h^3}.
\end{align*}
To estimate the expected value of service times, we consider the kernel $J_b$, which in this case equals
\[
J_b(t)=-\f{1}{2\pi\l_0} \int_{-\infty}^\infty 
(1-e^{-(c+i\omega)b}) e^{(c+i\omega)t+\f12 h^2 (c+i\omega)^2}\dif\omega = 
\f{1}{\sqrt{2\pi}\l_0h}\Big(
e^{-\f{1}{2h^2}(b-t)^2}-e^{-\f1{2h^2}t^2}\Big).
\]
This can be seen as the difference of two Gaussian kernels: one centered at $b$ and one at 0. 
As $h\downarrow0$, it converges to 
the difference of Dirac delta functions centered on $b$ and $0$. 
In other words, for $h\downarrow0$, the estimator of the mean service time is 
\[
\tilde{\mu}_G = \lim_{h\downarrow0} \f1{n\l_0} \sum_{i=1}^n  \int_0^\infty J_b(t)X_i(t)\dif t = \f1{n\l_0} \sum_{i=1}^n X_i(b).
\]
In this special case, it is also obvious how such an estimator could be obtained directly by observing that 
\[
\E_G[X(b)-X(0)] = \int_0^b \l(t-u)\bar{G}(u)\dif u = \l_0 \int_0^b \bar{G}(u)\dif u.
\]


\end{example}

\begin{example}[High/low switching]
\label{ex:onoff}
Suppose that the arrival rate repetitively switches between a high and low arrival rate $\l_0,\l_1$ 
after each time unit. Mathematically we can write this as
\[
\l(t) = \l_0 \sum_{j=0}^\infty {\bf 1}_{(2j,2j+1]}(t) + \l_1 \sum_{j=0}^\infty {\bf 1}_{(2j+1,2j+2]}(t).
\]
Elementary calculus yields 
$\widehat\l(z) = (\l_0+\l_1 e^{-z})z^{-1}(1+e^{-z})^{-1}$.
In the special case that the low rate is equal to zero, the kernels $L_h$ and $J_b$ can be calculated explicitly. 
Suppose $\l_0>0$ and $\l_1=0$, then 
\begin{equation*}
L_h(t) = \frac{1}{ \sqrt{2 \pi } \l_0 h^3}
e^{-\frac{(t+1)^2}{2 h^2}} \left(1+t+t e^{\frac{2 t+1}{2
   h^2}}\right),
\end{equation*}
from which the estimator of $G(x_0)$ follows. Furthermore, it can be calculated that
\[
J_b(t) = \f{1}{\sqrt{2\pi}\l_0 h}\Big(e^{-\f{t^2}{2h^2}}+e^{-\f{(t-1)^2}{2h^2}}-e^{-\f{(t-(1+b))^2}{2h^2}} - e^{-\f{(t-b)^2}{2h^2}}\Big).
\]
Following the same logic as in the previous example and by shifting the kernel $J_b$ to the left by 1 unit to get rid of bias, we see that (as $h\downarrow0$)
\begin{equation}
\label{eq:estmuonoff}
\tilde{\mu}_G = \f1{n\l} \sum_{i=1}^n (X_i(b-1)+X_i(b)).
\end{equation}
\end{example}

\begin{example}[Linearly increasing arrival rate]
\label{ex:linear}
Suppose that
$\l(t) = \l_0 t \,{\bf 1}_{[0,\infty)}(t)$; then 
$\widehat{\lambda}(z)=\lambda_0 z^{-2}$ and therefore 
\[
L_h(t) = \frac{1}{\sqrt{2 \pi } \lambda_0h^5}e^{-\frac{t^2}{2 h^2}} \big(t^2-h^2\big).
\]
Furthermore, with similar calculations as before, we find 
\[
J_b(t) = -\frac{1}{{\sqrt{2 \pi }
   h^3}}
   e^{-\frac{b^2+t^2}{2 h^2}} \Big(t e^{\frac{b^2}{2
   h^2}}+(b-t) e^{\frac{b t}{h^2}}\Big).
\]
In other words, $J_b$ is a differentiation kernel 
centered around $b$, minus a differentiation kernel centered at 
$0$. 
It can also be argued intuitively that differentiation plays a role for the estimator. 
Indeed, 
note 
that 
\[
\f{\dif}{\dif b} \E X(b) = \f{\dif}{\dif b}\left[\int_0^b (b-u)\bar{G}(u)\dif u\right] 
= \int_0^b \bar{G}(u)\dif u,
\]
and  for $b$ large the last integral is an approximation for $\mu_G$.
%
%
\end{example}
\subsubsection*{Numerical evaluation of kernels $L_h$ and $J_b$}
For some arrival rates $\l(\cdot)$ 
there is no analytic form for $L_h$ and $J_b$.
In that case, one needs to resort to numerical inversion. 
Following \cite{Durbin}, we discretize the Bromwich integral 
with a trapezoidal rule. 
Although it is assumed in \cite{Durbin} that 
the inverse transform is zero on 
the negative half-line,  the algorithm still works 
if this assumption does not hold. For completeness, we provide the 
following inversion formula ($L_h$ and $\widehat{L}_h$ can be replaced by $J_b$ and $\widehat {J}_b$, respectively):
\begin{align*}
L_h(t)=\f{2e^{ct}}{\tilde{T}}\Bigg(\f12 
{\rm Re}(\widehat{L}_h(c))&+
\sum_{k=1}^{n_{\text{max}}} 
{\rm Re}\left(\widehat{L}_h\big(c+i{k\pi}/{\tilde{T}}\big)\right)
\cos(k\pi t/\tilde{T})\\
&-\sum_{k=0}^{n_{\text{max}}} 
{\rm Im}\left(\widehat{L}_h\big(c+{ik\pi}/\tilde{T}\big)\right)
\sin(k\pi t/\tilde{T})\Bigg),
\end{align*}
where the error is controlled by the parameters $n_{\textrm{max}}, c, \tilde{T}$ and the machine precision. 
When applying this formula it is important that $c$ is chosen in the strip of convergence of $\widehat{L}$. 
For example, if $\widehat{L}(z)$ is defined for $z$ such that ${\rm Re}(z)<0$, we take $c$ small but 
$c<0$.  We stuck to the guideline that $|c\tilde{T}|\approx30$. 
In principle, picking $c$ closer to the edge of the strip should improve accuracy. 
However, when $c$ is too close to the edge of the strip, 
the solution will become numerically instable. For more details on picking 
the right parameters and bounds on the error, cf.\ \cite{Durbin}.

\subsection{Adaptive scheme for  bandwidth selection}
The bandwidth $h$ controls the trade-off between bias and variance when estimating the distribution function: 
lower $h$ corresponds to high variance and low bias, and vice versa. 
In applications, 
smoothness of $G$ is unknown, so the bandwidth choice given in  Theorem~\ref{th:upper-bound}
is not feasible.  In our simulations we implemented a variant of
Lepski's \cite{Lepski} adaptive procedure  proposed in \cite{GN}.
\par
The adaptive bandwidth scheme is as follows:
\begin{itemize}
\item[1.]Pick a minimum bandwidth $h_{\text{min}}>0$, and define $h_i:=(1+\alpha)^i h_{\text{min}}$ for $\alpha>0$;
\item[2.] Estimate the variance of $\tilde{G}_{h_i}(x_0)$ and denote the estimate by $v_{h_i}^2$;
\item[3.]Define $I_{h_i}$ as the interval $\big[\tilde{G}_{h_i}(x_0)- 
2\kappa v_{h_i},\tilde{G}_{h_i}(x_0)+ 2\kappa v_{h_i}\big]$, with $\kappa=\f14 \sqrt{\ln(n)}$;
\item[4.] The estimator will be the middle of the last interval $I_{j^*}$ that is nonempty, i.e. 
\[
j^*:=\max\Big\{j\;:\;\bigcap_{i=1}^j I_{h_i}\neq\emptyset\Big\}.
\]
\end{itemize}
\par 
Note that the variance of $\tilde{G}_h(x_0)$ is given by
\begin{align*}
 {\rm var}_G \big[\tilde{G}_h(x_0)\big] = 
 \frac{1}{n} \int_0^T \int_0^T L_{h}(t-x_0) \overline{L_h (\tau-x_0)} R(t, \tau) \rd t\rd \tau,
 \end{align*}
 where $R(t, \tau):={\rm cov}_G [X(t), X(\tau)]$ is determined by (\ref{eq:cov}).
We discretize the double integral 
and estimate the covariance function by its empirical counterpart.
In particular, 
we define a uniform time grid $\{t_i:1\leq i\leq N\}$ with 
grid size $\Delta$, a vector $L\in\R^N$ of which $i$th  component is 
$L_h(t_i-x_0)$, a matrix $X\in\R^{N\times n}$, 
such that $X_{ij}:=X^{(i)}(t_j)$, and a covariance matrix 
$R\in\R^{N\times N}$, such that element $ij$ gives the sample covariance between 
$X(t_i)$ and $X(t_j)$. More precisely,
\[
R = \f{1}{n-1}\sum_{i=1}^n (X_{i\cdot}-\bar{X})(X_{i\cdot}-\bar{X})^{\textrm{T}},
\]
where $X_{ij}=X_i(t_j)$ and $\bar X$ is such that 
$\bar{X}_j = \f1n \sum_{i=1}^n X_{ij}$ for $1\leq j \leq N$.
With this notation, the variance of the estimator is approximated by
\begin{align*}
 v_{h_i}^2=  \f1{n-1} \sum_{k=1}^N\sum_{\ell=1}^N L_{h_i}(t_k-x_0) \overline{L_h(t_\ell-x_0)} R(t_k,t_\ell) \Delta^2 = \f{\Delta^2}{n-1} L^* R L.
\end{align*}

\subsection{Simulation results for estimating the distribution function}
Now we present a small
simulation study on some particular examples. 
The goal is to verify consistency and compare accuracies 
in two cases where $\l(\cdot)$ is varied and 
$G$ is kept constant. 
Consider the following examples:
\begin{itemize}
\item[1a.] Take $G(x)=1-e^{-x}$, $\lambda(t) = 10(1+\cos(t))$, for $t\geq0$. Set parameters $h_{\textrm{min}}=0.025$, $\alpha=0.25$;
\item[1b.] Take $G(x)=1-e^{-x}$, $\lambda(t) = 10 t$, for $t\geq0$. Set parameters $h_{\textrm{min}}=0.05$, $\alpha=0.15$.
\end{itemize}

\pgfplotsset{width=7.5cm} 
\pgfplotsset{yticklabel style={text width=3em,align=right}}
\begin{figure}[ht!]
\begin{subfigure}{.49\textwidth}
\includegraphics{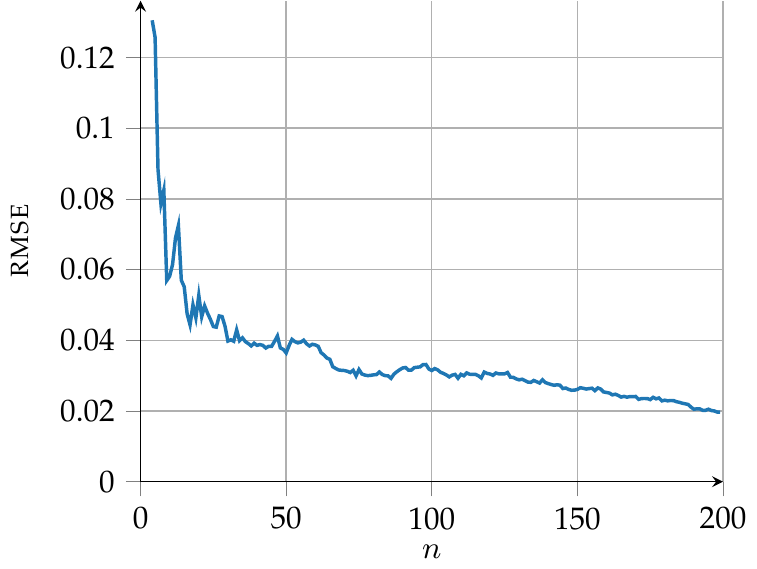}
\end{subfigure}
\begin{subfigure}{.49\textwidth}
\includegraphics{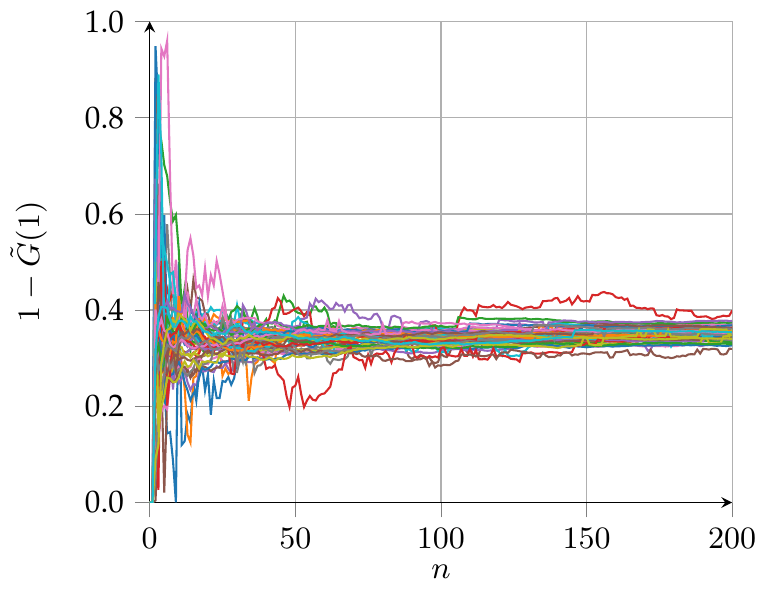}
\end{subfigure}
\begin{subfigure}{.49\textwidth}
\includegraphics{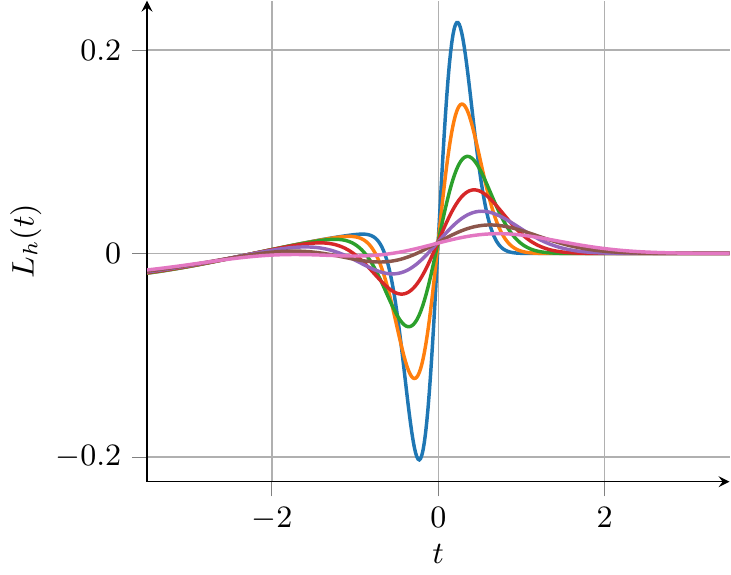}
\end{subfigure}
\begin{subfigure}{.49\textwidth}
\includegraphics{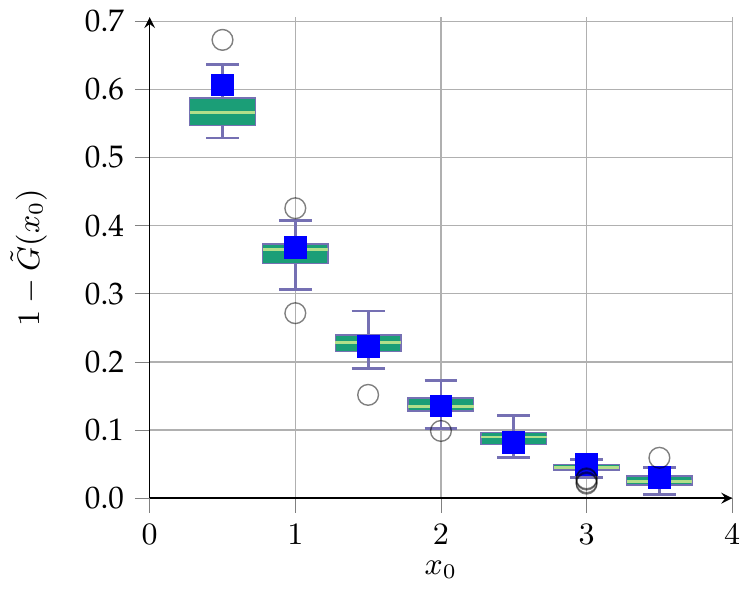}
\end{subfigure}
\begin{center}
\begin{tabular}{l|lllllll}
$x_0$ & 0.5&1.0&1.5&2.0&2.5&3.0&3.5\\
\hline
$1-G(x_0)$ & 0.607 & 0.368 & 0.223 & 0.135 & 0.0821 & 0.0414 & 0.0302\\
Mean $1-\tilde G(x_0)$ & 0.570 & 0.361 & 0.228 & 0.136 & 0.0976 & 0.0441 & 0.0261\\
St. dev. & 0.028&0.025&0.021&0.015&0.0140&0.0076&0.0097
\end{tabular}
\end{center}
\caption{\textit{Results corresponding to item 1a. 
Top left: plot of the root mean squared error vs the number of observations 
$n$, averaged over 50 runs. Top right: 
50 runs for $n=0$ to $n=200$, for $x_0=1$. 
Bottom left: $L_h$ for $h\in[0.23,0.89]$. 
More precisely: $h=0.025\cdot(1.25)^q$, with $q=10,11,\ldots,16$. 
Smaller $h$ corresponds to steeper bumps. Bottom right: boxplot of $50$ runs with $n=200$. 
The table corresponds to the boxplot.}}
\label{fig:hatGcos}
\end{figure}

\pgfplotsset{width=7.5cm} 
\pgfplotsset{yticklabel style={text width=3em,align=right}}
\begin{figure}[ht!]
\begin{subfigure}{.49\textwidth}
\includegraphics{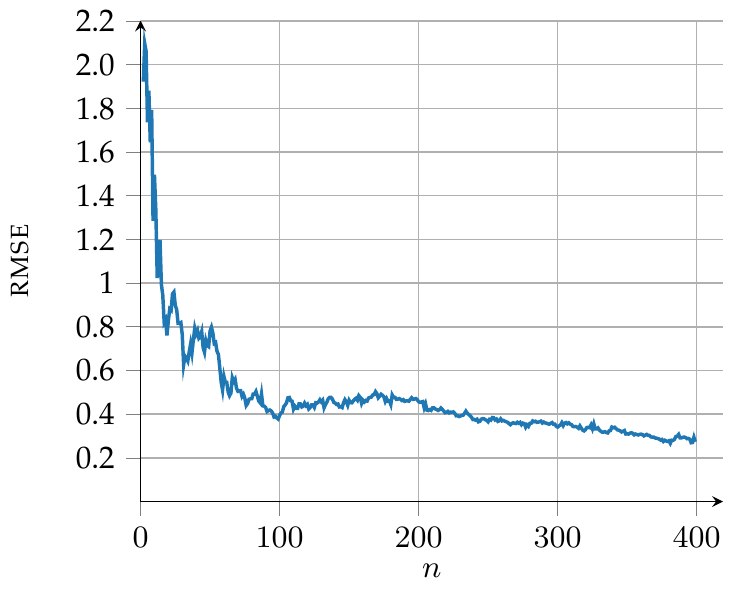}
\end{subfigure}
\begin{subfigure}{.49\textwidth}
\includegraphics{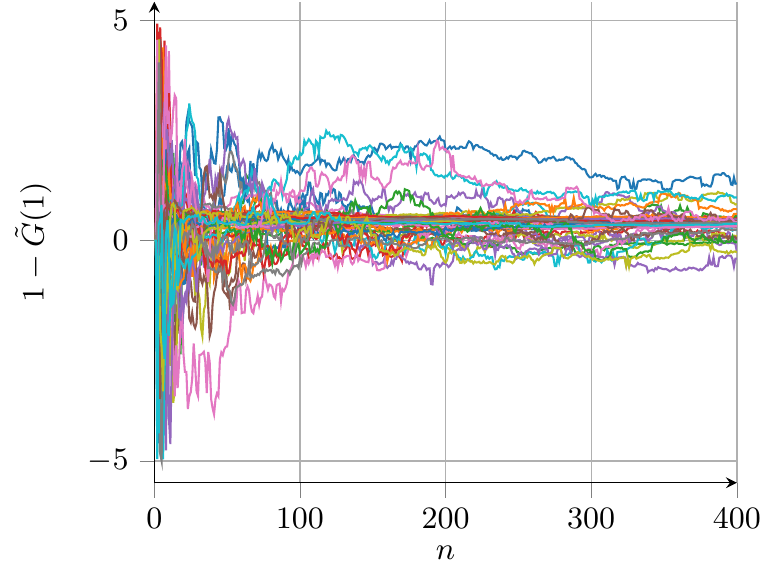}
\end{subfigure}
\begin{subfigure}{.49\textwidth}
\includegraphics{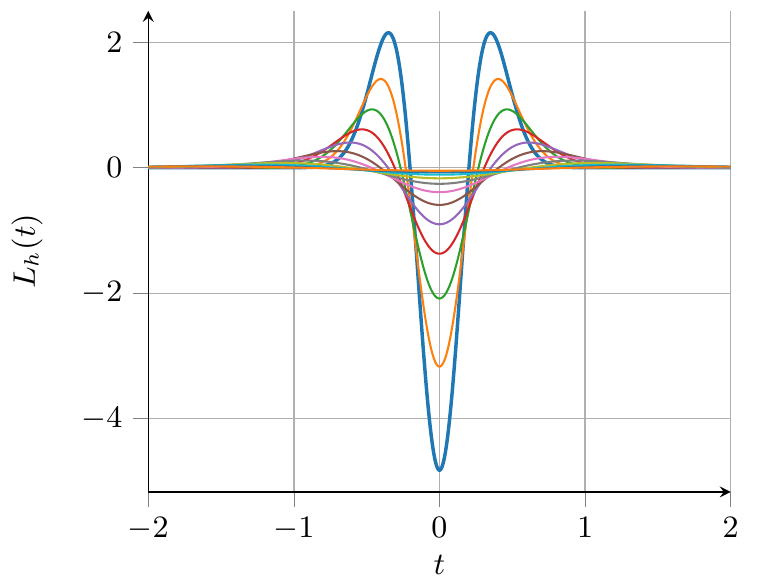}
\end{subfigure}
\begin{subfigure}{.49\textwidth}
\hspace{1.0cm}
\begin{tabular}{l|llll}
$n$& 100&200&300&400\\
\hline
Mean			 & 0.43 & 0.42 & 0.35 &0.37\\
St.\ dev.        & 0.39 & 0.45 & 0.34 &0.28\\
\end{tabular}
\end{subfigure}
\caption{\textit{Results corresponding to Case 1b. 
Top left: root mean squared error vs $n$. Top right: 50 runs for 
$n=0$ to $n=400$, for $x_0=1$. Bottom left: $L_h$ for $h\in[0.23,0.94]$. 
More precisely: $h=0.05\cdot(1.15)^q$, with $q=10,11,\ldots,21$. 
Smaller $h$ corresponds to steeper bumps. Bottom right: descriptive statistics of estimator $1-\tilde{G}(1)$.}}
\label{fig:hatGlin}
\end{figure}

Figure \ref{fig:hatGcos} corresponds to Case 1a, and 
Figure \ref{fig:hatGlin} corresponds to Case 1b. In these figures 
(from left to right, top to bottom), we plotted the \textsc{rmse} as function of $n$, 
which is based on 50 runs of $1-\tilde{G}(1)$, 
the estimator of $1-G(1)=e^{-1}\approx 0.368$, which are plotted 
next to it. Furthermore, we plotted the 
kernel $L_h(t)$, for several values of $h$ as indicated in the captions. 
For Case 1, we also included a boxplot to show the performance of these 
50 runs for $x_0$ ranging from $0.5$ to $3.5$.
\par
From the figures we note the following. Firstly, 
the \textsc{rmse} is clearly smaller in Case 1a compared to 1b. 
This was to be expected because 
$\gamma=0$ in Case 1a, and $\gamma=1$ in Case 1b, which leads 
to inferior worst case minimax convergence rates in theory. 
Note that in Case 1b the estimators are often not even within 
$[0,1]$. In Case 1a we see some slight bias for smaller $x_0$: it appears 
that estimation for $x$ small in this case is relatively hard, which could be 
due to the fact that $G(x)$ is steeper for small $x$. 
The $L_h$ kernel in Case 1a seems quite similar to the 
standard Gaussian differentiation kernel that would be used in the case 
of constant $\lambda(\cdot)$  (when $\l(\cdot)\equiv\lambda_0$ is constant, 
the derivative of $\int_0^t \l_0 \bar{G}(s)\dif s$ equals $\l_0 \bar{G}(t)$). 
However, there is an additional `bump' that compensates for the fact that 
$\l(\cdot)$ is a cosine function, rather than a constant.

\subsection{Simulation results for estimating the  mean}
We consider the following numerical settings.
\begin{itemize}
\item[2a.] Take $\l(t)=1+\sin(t)$ and $G(x)=1-e^{-x}$, $x_1=-1$, $h=0.08$ 
(as small as possible so that the numerical inversion algorithm is still stable, 
with parameters $\tilde T=30$, $c=-1$, $b=25$. 
\item[2b.] Let $\l$ be the on-off rate of Example \ref{ex:onoff} and 
suppose that service times are uniformly distributed on $[0,2]$, such 
that $\mu_G=1$. We take $b=25$ in estimator \eqref{eq:estmuonoff}. 
\end{itemize}

In Case 2a, of which the results are shown in 
Figure \ref{fig:case2a}, we see that the kernel is 
approximately a delta function at $b+x_1=24$ 
(as in the constant $\lambda(\cdot)$ case), with some adjustment that takes 
care of the fact that the arrival rate is a sine function. 
The histogram shows the distribution of $\hat\mu_G$, each 
based on a single observation. Note that the distribution is approximately 
normal, with some skew to the left, with mean $1.03$ and 
standard deviation $0.73$. Although the standard deviation seems quite 
large, one must take into account that the 
kernel $J_b$ is almost zero outside of $[15,25]$. 
During those 10 time units there are approximately 10 arrivals. 
Note that the kernel $J_b$ in Case 2a is again quite similar to 
kernel $J_b$ in case of constant $\l(\cdot)$ 
(as in Example \ref{ex:const}): there is a Dirac delta function 
at $b+x_1$, and in addition to that there is a wave preceding the Dirac delta function 
to somehow compensate for the fact that $\lambda$ is not constant. 
\par 
In Case 2b we used the explicit kernels 
as calculated in Example~\ref{ex:onoff}. The bias turns out to be 
negligible, and the standard deviation is approximately $\f1{\sqrt{n}}$ after $n$ observations.

\pgfplotsset{width=7.5cm} 
\pgfplotsset{yticklabel style={text width=3em,align=right}}
\begin{figure}
\begin{subfigure}{.49\textwidth}
\includegraphics{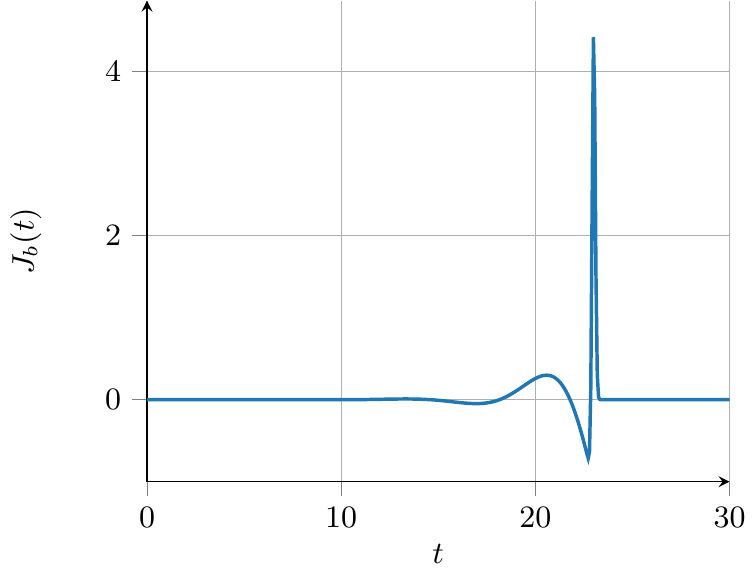}
\end{subfigure}
\begin{subfigure}{.49\textwidth}
\includegraphics{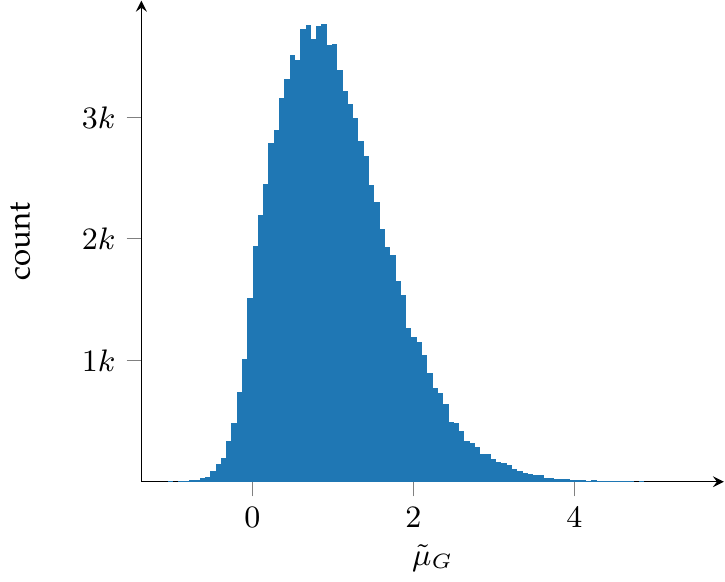}
\end{subfigure}
\caption{\textit{Results corresponding to Case 2a. 
On the top left we plotted the kernel $J_b$ and on the top right 
we plotted $\tilde{\mu}_G$ based on a single observation, 
for $10^5$ simulation runs.
}}
\label{fig:case2a}
\end{figure}

\pgfplotsset{width=7.5cm} 
\pgfplotsset{yticklabel style={text width=3em,align=right}}
\begin{figure}
\begin{center}
\begin{subfigure}{.49\textwidth}
\includegraphics{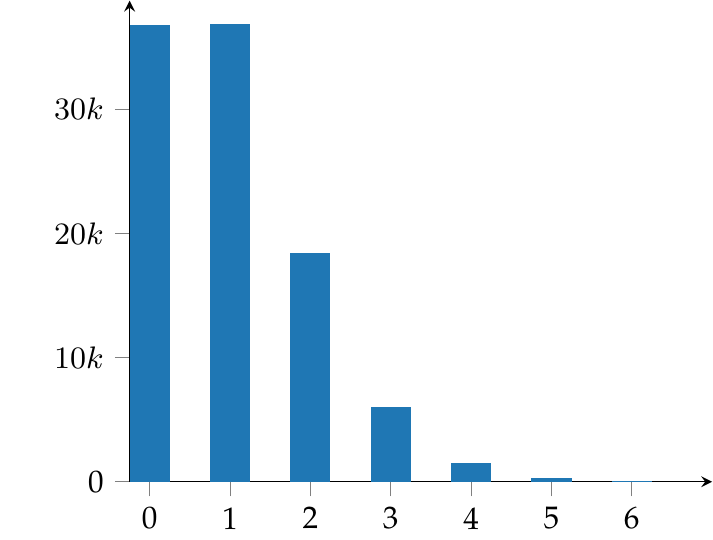}
\end{subfigure}
\end{center}
\caption{\textit{Results corresponding to Case 2b. 
We plotted the histogram of estimators based on a single observation. 
All estimates based on a single observation are obviously integer valued. The mean of this sample is 0.999, with a standard deviation of 0.998.}}
\label{fig:case2b}
\end{figure}


\section{Proofs}\label{sec:proofs}

\subsection{Proof of Theorem~\ref{prop:1}}
The proof is similar to the one of Proposition~1 in \cite{Goldenshluger}.
By conditioning on $\{\tau_j, j\in \bZ\}$ we obtain that
\[
 \bE_{G} \exp\Big\{\sum_{i=1}^m \theta_iX (t_i) \Big\}= \bE_{G} \exp\Big\{\sum_{j\in \bZ}
f(\tau_j)\Big\},
\]
where 
\[
 f(x)=\ln\Big(1+\sum_{k=1}^m \Big[e^{\sum_{i=1}^k \theta_i
{\bf 1}(x\leq t_i)} -1\Big] \bP_{G}\{\sigma_1\in I_k(x)\}\Big),
\]
$I_k(x)=(t_k-x, t_{k+1}-x]$, $k=1,\ldots, m-1$, and $I_m(x)=(t_m-x, \infty)$.
Then Campbell's formula yields
\[
 \bE_{G} \exp\Big\{\sum_{j\in \bZ}
f(\tau_j)\Big\}=\exp\Big\{\int_{-\infty}^\infty [e^{f(x)}-1]\lambda(x)\rd x\Big\},
\]
and our current goal is to evaluate the integral on the right hand side of the 
last equality. Denote this integral by $S_m$:  
\begin{eqnarray}
S_m &:=&\int_{-\infty}^\infty [e^{f(x)}-1] \lambda(x) \rd x 
\nonumber
\\
&=& \sum_{k=1}^{m-1} \int_{-\infty}^\infty 
 \big(\exp\big\{\sum_{i=1}^k \theta_i {\bf 1}(x \leq t_i)\big\}-1\big)
\big[\bar{G}(t_k-x)- \bar{G}(t_{k+1}-x)\big]\lambda(x)\rd x
\nonumber
\\
&&+ \int_{-\infty}^\infty \big(\exp\big\{\sum_{i=1}^m
\theta_i {\bf 1}(x \leq t_i)\big\}-1\Big)
\bar{G}(t_m-x)\lambda(x)\rd x
=:\sum_{k=1}^{m-1} J_k + L_m.
\label{eq:Sm}
\end{eqnarray}
We derive (\ref{eq:formula}) by iterating (\ref{eq:Sm}); we have 
\begin{equation}\label{eq:S-m+1}
 S_{m+1}= S_m + J_m + L_{m+1} - L_m = S_1 +\sum_{k=1}^{m}[J_k+L_{k+1}-L_k]. 
\end{equation}
For any $k\geq 1$
\begin{eqnarray*}
 J_k&=& \Big(e^{\sum_{i=1}^k \theta_i}-1\Big)\int_{t_k-t_1}^\infty 
\big[\bar{G}(u)-\bar{G}(t_{k+1}-t_k+u)\big] \lambda(t_k-u)\rd u
\\
&&\;+\; \sum_{j=1}^{k-1} \Big(e^{\sum_{i=j+1}^k \theta_i}-1\Big)
\int_{t_k-t_{j+1}}^{t_k-t_j} \big[\bar{G}(u)-\bar{G}(t_{k+1}-t_k+u)\big] \lambda(t_k-u)\rd u,
\\
L_k&=& \Big(e^{\sum_{i=1}^k \theta_i}-1\Big) \int_{t_k-t_1}^\infty \bar{G}(u)\lambda(t_k-u)\rd u
\\
&&\;
+\;\sum_{j=1}^{k-1}
 \Big(e^{\sum_{i=j+1}^k \theta_i}-1\Big) \int_{t_k-t_{j+1}}^{t_k-t_j} \bar{G}(u)\lambda(t_k-u)\rd u;
\end{eqnarray*}
this yields
\begin{eqnarray*}
J_k-L_k &=& -\Big(e^{\sum_{i=1}^k \theta_i}-1\Big) \int_{t_k-t_1}^\infty \bar{G}(t_{k+1}-t_k+u)\lambda(t_k-u)\rd u
\\
&&\;-\;
\sum_{j=1}^{k-1} \Big(e^{\sum_{i=j+1}^k\theta_i} -1\Big) \int_{t_k-t_{j+1}}^{t_k-t_j}
\bar{G}(t_{k+1}-t_k+u) \lambda(t_k-u)\rd u 
\\
&=&  -\Big(e^{\sum_{i=1}^k \theta_i}-1\Big) \int_{t_{k+1}-t_1}^\infty \bar{G}(u)\lambda(t_{k+1}-u)\rd u
\\
&&\;-\;
\sum_{j=1}^{k-1} \Big(e^{\sum_{i=j+1}^k\theta_i} -1\Big) \int_{t_{k+1}-t_{j+1}}^{t_{k+1}-t_j}
\bar{G}(t_{k+1}-t_k+u) \lambda(t_k-u)\rd u. 
\end{eqnarray*}
Then we obtain 
\begin{eqnarray*}
 J_k-L_k+L_{k+1} \;=\; \Big(e^{\theta_{k+1}}-1\Big) H(t_{k+1})
 + \Big(e^{\theta_{k+1}}-1\Big)
\Big(e^{\sum_{i=1}^k\theta_i}-1\Big) H_{t_{k+1}-t_1,\infty} (t_{k+1}) 
\\
 + \Big(e^{\theta_{k+1}}-1\Big)
\sum_{j=1}^{k-1} \Big(e^{\sum_{i=j+1}^k \theta_i}-1\Big) H_{t_{k+1}-t_{j+1}, t_{k+1}-t_j}(t_{k+1}).
\\
\;=\; \Big(e^{\theta_{k+1}}-1\Big) \Big[ H(t_{k+1}) +
\sum_{j=0}^{k-1} \Big(e^{\sum_{i=j+1}^k \theta_i}-1\Big) H_{t_{k+1}-t_{j+1}, t_{k+1}-t_j}(t_{k+1})\Big],
\end{eqnarray*}
where by convention we put $t_0=-\infty$.
Taking into account that $S_1= (e^{\theta_1}-1) H(t_1)$ and using (\ref{eq:S-m+1}) we complete the proof.

\subsection{Auxiliary results}
First we state two auxiliary results that are used in the subsequent proofs.

\begin{lemma}\label{lem:covariance}
 Let $\{X(t), t\geq 0\}$ be the queue--length process of the 
 $M_t/G/\infty$ queue with $G\in \sG(M)$, and let
 $R(t,\tau):={\rm cov}_G[X(t), X(\tau)]$.
 Suppose that Assumption~\ref{as:lambda-1} holds.
\begin{itemize}
\item[(a)] If $\sigma_\lambda>0$ then for any $\nu\geq 0$ 
\begin{equation}\label{eq:R-sigma>0}
\int_0^\infty R(t,\tau) \rd \tau \leq 2M\lambda_0 e^{(\sigma_\lambda+\nu)t}(\sigma_\lambda+\nu)^{-1}.
\end{equation}
\item[(b)] If $\sigma_\lambda=0$ then 
\begin{equation*}
 \int_0^\infty R(t,\tau) \rd \tau \leq 3M\lambda_0 (a_1 + a_2 t^p).
\end{equation*}
\end{itemize}
\end{lemma}
\begin{proof}
 By Corollary~\ref{cor:1},
\begin{equation}\label{eq:R}
R(t,\tau):= {\rm cov}_G\big[X(t),X(\tau)] = \int_{\tau-t}^\tau \bar{G}(u)\l(\tau-u)\dif u,\;\;\;\
\forall \tau\geq t\geq 0.
\end{equation}
(a). First, assume $\sigma_\lambda>0$. 
By (\ref{eq:lambda-growth}) 
we have for any $\nu\geq 0$
\begin{eqnarray}
 \int_0^\infty R(t, \tau) {\bf 1}(\tau\geq t) \rd \tau =
 \int_0^\infty {\bf 1}(\tau\geq t)\int_{\tau-t}^\tau \bar{G}(u) \lambda(\tau-u)\rd u \rd \tau 
 \nonumber
 \\
\leq 
 \int_0^\infty  \bar{G}(u) \rd u \int_0^t \lambda(u)\rd u
 \leq M\lambda_0 (\sigma_\lambda+\nu)^{-1} e^{(\sigma_\lambda +\nu)t},
\label{eq:R-1}
 \end{eqnarray}
 and 
 \begin{multline}
 \int_0^\infty R(t, \tau) {\bf 1}(\tau < t) \rd \tau = \int_0^\infty {\bf 1}(\tau<t) \int_{t-\tau}^t 
 \bar{G}(u) \lambda(t-u) \rd u \rd \tau
 \\
 =  \int_0^t u \bar{G}(u) \lambda(t-u)\rd u
 \leq e^{(\sigma_\lambda+\nu)t }\int_{0}^t u e^{-(\sigma_\lambda+\nu)u} \bar{G}(u)\rd u
 \leq M\lambda_0 (\sigma_\lambda+\nu)^{-1} e^{(\sigma_\lambda+\nu)t}.
\label{eq:R-2}
 \end{multline}
Summing up (\ref{eq:R-1}) and (\ref{eq:R-2}) we obtain (\ref{eq:R-sigma>0}).
\par  
(b). Now consider the case $\sigma_\lambda=0$.  Using (\ref{eq:lambda-growth-2}) and (\ref{eq:R})
we obtain
\begin{align*}
 \int_0^\infty R(t, \tau) {\bf 1}(\tau\geq t) \rd \tau &=  
 \int_0^\infty  \int_0^\infty \bar{G}(u) {\bf 1}\{t\vee u \leq \tau \leq t+u\} \lambda(\tau-u)\rd \tau \rd u
 \\
 &\leq \lambda_0 a_1 \int_{0}^\infty \bar{G}(u) (t\wedge u) \rd u + 
 \lambda_0 a_2 \int_0^\infty \bar{G}(u) \int_{t\vee u}^{t+u} (\tau-u)^p \rd \tau \rd u
 \\
 &\leq \lambda_0 M a_1  + 
 \frac{\lambda_0 a_2}{p+1}\bigg[ \int_0^t \bar{G}(u)[t^{p+1} - (t-u)^{p+1}]\rd u + t^{p+1}\int_t^\infty 
 \bar{G}(u) \rd u\bigg].
 \end{align*}
 Moreover, since $G\in \sG(M)$, it holds that $t^{p+1}\int_t^\infty  \bar{G}(u) \rd u \leq Mt^p$ and 
 \begin{eqnarray*}
 \int_0^t \bar{G}(u)[t^{p+1} - (t-u)^{p+1}]\rd u \leq t^{p+1} \int_0^t \bar{G}(u) (p+1)\tfrac{u}{t}\rd u
 \leq (p+1) M t^p,
 \end{eqnarray*}
 where in the first inequality we have used that $1-[1-(u/t)]^p \leq (p+1)u/t$ for $0\leq u\leq t$.
This yields
\[
\int_0^\infty R(t, \tau) {\bf 1}(\tau\geq t) \rd \tau \leq 2M\lambda_0 (a_1+a_2t^p).
\]
 Similarly,
 \begin{eqnarray*}
 \int_0^\infty R(t, \tau) {\bf 1}(\tau < t) \rd \tau =  \int_0^t u \bar{G}(u) \lambda(t-u)\rd u \leq M\lambda_0 (a_1  + a_2 t^{p}).
\end{eqnarray*}
Summing up the  last two inequalities we complete the proof. 
\end{proof}
\begin{lemma}\label{lem:Gamma-function}
Let $\delta>0$ and $p\geq 0$; then
\[
 \int_T^\infty e^{-\delta t} t^p \rd t \leq 2\delta^{-2} T^{p-1}e ^{-\delta T}, 
\]
provided that $\delta T\geq 2p$. 
\end{lemma}
\begin{proof}
The statement is an immediate consequence of Corollary~2.5  in \cite{Borwein09}. 
\end{proof}

\subsection{Proof of Theorem~\ref{th:upper-bound}}
 Throughout the proof 
 $c_1, c_2, \ldots$ stand for positive constants that may depend on $\beta$, $\gamma$, $k_0$,  and $p$ only.
 In particular,  they 
do not depend on $n$, $T$, and on the parameters $x_0$, $A$, $M$, $\sigma_\lambda$ and $\lambda_0$.
The proof is divided in steps. The numbering of constants at each step of the proof starts anew, so that
$c_1, c_2,\ldots$ are different on different appearances. 
\par 
1$^0$. 
First we show that under the premise of the theorem the estimator is well--defined.
Since ${\rm supp}(K)=[0,1]$, $\widehat{K}$ is
an entire function. 
For any $\sigma\in \bR$ 
\[
 \widehat{K}(\sigma+i\omega) = \int_{0}^1  K(t) e^{-(\sigma+i\omega) t}\rd t,
\]
and, by
condition~(K2), kernel $K$ is $r$ times continuously differentiable; hence  
repeatedly integrating by parts 
we obtain that for any $\omega$ and $\sigma$
\begin{equation}\label{eq:hatK}
 |\widehat{K}(\sigma+i\omega) | \leq c_1 e^{|\sigma|}(\sigma^2 + \omega^2)^{-r/2}.
\end{equation}
In addition, in view of~(K2),  for any $\omega$ and $\sigma$ 
\begin{equation}\label{eq:hatK-1}
|\widehat{K}(\sigma+i\omega)|\leq c_1 e^{|\sigma|}.
\end{equation}
\par 
By Assumption~~\ref{as:lambda}, $\widehat{\lambda}(z)$ does not have zeros in the half--plane 
$\{z\in \bC: {\rm Re}(z)>\sigma_\lambda\}$. This implies that  $\widehat{K}(zh)/\widehat{\lambda}(-z)$ is an analytic 
function in $\{z\in \bC: {\rm Re}(z)<-\sigma_\lambda\}$, and   integration in (\ref{eq:L}) can be performed
on any vertical line ${\rm Re}(z)=s$ with $s<-\sigma_\lambda$. 
\par 
Put $s=-\sigma_\lambda-\delta$, where $\delta>0$ is a parameter to be specified.
In the subsequent proof  we   assume that $\delta<1/h$ and $h<1$; 
this does not lead to  loss of generality because $h$ will be chosen
tending to zero as $n\to \infty$, and $\delta$ will always satisfy $\delta<1/h$ for large $n$.
Thus, 
for $s=-\sigma_\lambda-\delta$ it follows from (\ref{eq:L}) and   Assumption~\ref{as:lambda} that
\begin{eqnarray}
 |L_{h}(t)| &\leq& \frac{e^{-(\sigma_\lambda+\delta)t}}{2\pi}
 \int_{-\infty}^\infty \frac{|\widehat{K}((-\sigma_\lambda-\delta+i\omega)h)|}{|\widehat{\lambda}(\sigma_\lambda+\delta-i\omega)|}\rd 
 \omega 
 \nonumber
 \\
 &\leq& \frac{e^{-(\sigma_\lambda+\delta) t}}{2\pi \lambda_0 k_0}
 \int_{-\infty}^\infty |\widehat{K}((-\sigma_\lambda-\delta+i\omega)h)| (\delta^2+\omega^2)^{\gamma/2}  \rd 
 \omega.
 \nonumber
\end{eqnarray}
We continue bounding this quantity: 
\begin{align}
 &|L_{h}(t)| \nonumber
 \\
\nonumber &\leq  \frac{c_2e^{-(\sigma_\lambda+\delta)t}}{\lambda_0} \left[ \delta^\gamma 
 \int_{|\omega|\leq \delta} |\widehat{K}((-\sigma_\lambda-\delta+i\omega)h)|\rd \omega  + 
 \int_{|\omega|>\delta} |\widehat{K}((-\sigma_\lambda-\delta +i\omega)h)| \,
 |\omega|^\gamma \rd \omega
 \right]
 \\
\nonumber& \leq 
 \frac{c_3 e^{-(\sigma_\lambda+\delta)t}}{ \lambda_0} \left[ \delta^{\gamma+1} e^{(\sigma_\lambda+\delta)h} + h^{-\gamma-1}
 \int_{|\xi|>\delta h} |\widehat{K}(-(\sigma_\lambda+\delta)h+i\xi)|\,|\xi|^{\gamma}\rd \xi\right]
 \\
 &\leq  c_4 e^{-(\sigma_\lambda+\delta)(t-h)}(\lambda_0h^{\gamma+1})^{-1},
 \label{eq:L-tail}
\end{align}
where 
in the first inequality above we use that 
$(\delta^2+\omega^2)^{\gamma/2} \leq 2^{(\frac{\gamma}{2}-1)_+}(\delta^\gamma + \omega^\gamma)$, 
and the last inequality follows from (\ref{eq:hatK}), (\ref{eq:hatK-1}) and  $r>\gamma+1$. 
\par
Since $G\in \bar{\sH}_{\beta, x_0} (A, M)$, it holds  that 
$H(t)\leq M \lambda_0 e^{\sigma_\lambda t}$ when $\sigma_\lambda>0$, 
and 
$H(t)\leq M \lambda_0 (a_1+a_2t^p)$ when $\sigma_\lambda=0$ [cf. (\ref{eq:H(t)})]. 
Combining this with (\ref{eq:L-tail}) we obtain 
\[
 \int_0^\infty |L_h(t-x_0)| H(t)\rd t \leq 
 c_5\kappa_1(\delta) 
 Me^{(\sigma_\lambda+\delta)(x_0+h)}h^{-\gamma-1}
<\infty,
\]
where we put
\[
 \kappa_1(\delta):= \left\{\begin{array}{ll}
                                           \delta^{-1}, & \sigma_\lambda>0,
                                           \\
                                           a_1\delta^{-1}+a_2\delta^{-p-1}, & \sigma_\lambda=0.
                                          \end{array}
\right.
\]
Thus the estimator $\tilde{G}_h(x_0)$ is well defined, and conditions  (\ref{eq:conditions}) of Lemma~\ref{lem:L} hold.
\par\medskip
2$^0$. Now we establish an upper bound on the bias of $\tilde{G}_{h}(x_0)$.
By Lemma~\ref{lem:L} and conditions (K1)--(K2) we have 
\begin{align*}
 \bE_G \big[\tilde{G}_{h}(x_0)\big] &= 1-
 \bE_G \int_0^T  L_{h}(t-x_0) X(t) \rd t = 1 - \int_0^T L_{h}(t-x_0) H(t)\rd t
 \\
 &= \int_{-\infty}^\infty \frac{1}{h}K\bigg(\frac{x-x_0}{h}\bigg) G(x)\rd x + \int_T^\infty L_{h}(t-x_0)H(t)\rd t. 
\end{align*}
It follows from Taylor expansion of $G$ and  (K1)-(K2) that 
\begin{equation*}
 \bigg|\int_{-\infty}^\infty \frac{1}{h}K\bigg(\frac{x-x_0}{h}\bigg) G(x)\rd x  - G(x_0) \bigg|
 = \bigg|\int_{0}^1 K(u) [ G(x_0+uh)-G(x_0)]\rd u\bigg| \leq C_K Ah^\beta. 
\end{equation*}
Moreover, 
if $\sigma_\lambda>0$ then 
in view of (\ref{eq:L-tail}) 
\[
 \bigg|\int_T^\infty L_{h}(t-x_0)H(t)\rd t\bigg| \leq 
  c_1  Me^{(\sigma_\lambda+\delta)(x_0+h)}
  e^{-\delta T}\delta^{-1} h^{-\gamma-1}.
\]
If $\sigma_\lambda=0$, $p>0$ and $\delta T\geq 2p$ then by Lemma~\ref{lem:Gamma-function} 
\[
 \bigg|\int_T^\infty L_{h}(t-x_0)H(t)\rd t\bigg| \leq 
  c_2 Me^{\delta (x_0+h)}
  e^{-\delta T}h^{-\gamma-1} (a_1\delta^{-1}+a_2 \delta^{-2}T^{p-1}).
  \]
Thus, 
\[
 \bigg|\int_T^\infty L_{h}(t-x_0)H(t)\rd t\bigg| \leq 
  c_3 \kappa_2(\delta, T) Me^{(\sigma_\lambda+\delta)(x_0+h)}
  e^{-\delta T}h^{-\gamma-1},
\]
where 
\[
 \kappa_2(\delta, T):=\left\{\begin{array}{ll}
       \delta^{-1} & \sigma_\lambda>0 \;\;\hbox{and}\;\;
       \sigma_\lambda=0, p=0,\\
       a_1\delta^{-1}+ a_2\delta^{-2} T^{p-1}, & \sigma_\lambda=0, p>0.
                             \end{array}
\right.
\]
Combining these inequalities we obtain 
\begin{equation}\label{eq:bias}
 \Big|\bE_G \big[\tilde{G}_{h}(x_0)- G(x_0)\big]\Big| \leq c_4 \bigg[Ah^\beta  + 
 \kappa_2(\delta, T)Me^{(\sigma_\lambda+\delta)(x_0+h)} e^{-\delta T}
 h^{-\gamma-1}
\bigg].
\end{equation}
\par\medskip 
3$^0$. The next step is to bound the variance of $\tilde{G}_{h}(x_0)$.
Let $R(t,\tau):= {\rm cov}_G\big[X(t),X(\tau)]$;
then 
\begin{eqnarray}\label{eq:var-1}
 {\rm var}_G \big[\tilde{G}_h(x_0)\big] 
 &=& \frac{1}{n} \int_0^T \int_0^T 
 L_{h}(t-x_0) \overline{L_h (\tau-x_0)} R(t, \tau) \rd t\rd \tau
\nonumber
\\
 &\leq&  
 \frac{1}{n} \int_0^T \int_0^T |L_h (t-x_0)|^2 R(t, \tau)\rd t\rd \tau
 \nonumber
 \\
 &\leq & \f1n \int_0^\infty |L_h(t-x_0)|^2\int_0^\infty R(t,\tau)\dif\tau\dif t,
\end{eqnarray}
where the first inequality is due to Cauchy--Schwarz and symmetry of the covariance function in its arguments. 
Our current goal is to bound from above the expression on the right hand side of (\ref{eq:var-1}). 
We consider the cases $\sigma_\lambda>0$
and $\sigma_\lambda=0$ separately.
\par
(a). If $\sigma_\lambda>0$ then using statement~(a) of Lemma~\ref{lem:covariance} with $\nu=\sigma_\lambda+2\delta$, $\delta>0$ 
we obtain
\begin{eqnarray*}
 {\rm var}_G \big[\tilde{G}_h(x_0)\big] &\leq& \frac{\lambda_0 M}{n(\sigma_\lambda+\delta)} 
 \int_0^\infty |L_h(t-x_0)|^2 \,e^{2(\sigma_\lambda+\delta)t} \rd t
\\
&\leq& 
\frac{\lambda_0 M e^{2(\sigma_\lambda+\delta)x_0}}{4\pi^2 n(\sigma_\lambda+\delta)} 
 \int_{-\infty}^\infty \bigg|\frac{\widehat{K}(-(\sigma_\lambda + \delta)h+i\omega h)}
 {\widehat{\lambda}(\sigma_\lambda+\delta
 -i\omega)}
 \bigg|^2\rd \omega,
 \end{eqnarray*}
where in the last line we have used Parseval's identity.
By Assumption~\ref{as:lambda} and in  view of (\ref{eq:hatK}) and (\ref{eq:hatK-1})
\begin{align*}
 \int_{-\infty}^\infty \bigg|\frac{\widehat{K}(-(\sigma_\lambda+\delta) h+i\omega h)}{\widehat{\lambda}(\sigma_\lambda+\delta
 -i\omega)}
 \bigg|^2\rd \omega &\leq   \frac{1}{\lambda_0^2 k_0^2} 
 \int_{-\infty}^\infty |\widehat{K}(-\sigma h+i\omega h)|^2 \big[\delta^2+\omega^2\big]^\gamma \rd \omega 
\\
&\leq c_1 \lambda_0^{-2} e^{2\sigma h} \big[\delta^{2\gamma} h^{-1}  + h^{-2\gamma-1}\big]
\leq  c_2 \lambda_0^{-2} e^{2\sigma h}h^{-2\gamma-1},
 \end{align*}
 where we have taken into account that $\delta<1/h$.
This leads to the following bound on the variance 
\begin{equation}\label{eq:var>0}
 {\rm var}_G \big[\tilde{G}_h(x_0)\big] \leq  
 c_3 (\sigma_\lambda+\delta)^{-1}  Me^{2(\sigma_\lambda+\delta)(x_0+h)}(\lambda_0 n h^{2\gamma+1})^{-1}.
\end{equation}
\par
(b). Now consider the case $\sigma_\lambda=0$.
In this case 
(\ref{eq:var-1}) and Lemma~\ref{lem:covariance} show that 
\begin{eqnarray}\label{eq:V-1}
 {\rm var}_G [\tilde{G}_h(x_0)] \leq \frac{c_{4}\lambda_0M}{n} \int_0^\infty |L_h(t-x_0)|^2 (a_1+a_2t^{p})\rd t,
\end{eqnarray}
and our current goal is to bound the integral on the right hand side.
\par 
It follows from Assumption~\ref{as:lambda-1}(b) that 
$\widehat{\lambda}(z)$ is bounded in its region of  analyticity; hence the limit 
$\lim_{\sigma\downarrow  0}
\widehat{\lambda}(\sigma+i\omega) = \widehat{\lambda}(i\omega)$ 
exists for almost all $\omega$ (see, e.g., \cite[Chapter~19]{Hille}). Moreover, 
Assumption~\ref{as:lambda} implies that 
$\widehat{\lambda}(z)$ does not vanish on the imaginary axis. 
Therefore
%
$|\widehat{\lambda}(\sigma+i\omega)|^{-1}$ is 
finite for all $\sigma\geq 0$ and $\omega\in \bR$, and function $\widehat{K}(zh)/\widehat{\lambda}(-z)$ 
can be analytically continued to the imaginary axis. Thus the 
integral in the definition (\ref{eq:L}) of $L_h$ can be computed over the line 
${\rm Re}(z)=0$, and
by 
Parseval's identity and Assumption~\ref{as:lambda} we obtain
\begin{align*}
 \int_{-\infty}^\infty |L_h(t-x_0)|^2 \rd t &= \frac{1}{4\pi^2}\int_{-\infty}^\infty \bigg|
 \frac{\widehat{K}(i\omega h)}{\widehat{\lambda} (-i\omega)}\bigg|^2\rd \omega
 \nonumber
 \\
 &\leq \frac{1}{4\pi^2} \int_{-\infty}^\infty |\widehat{K}(i\omega h)|^2 |\omega|^{2\gamma}\rd \omega
\leq c_{5}h^{-2\gamma-1},
\label{eq:V-2}
 \end{align*}
where the last inequality holds in view of (\ref{eq:hatK}), (\ref{eq:hatK-1})  and $r>\gamma+1$.
Furthermore, 
for any real  $\delta>0$ 
\begin{align*}
 &\int_0^\infty |L_h(t-x_0)|^2 t^{p}\rd t 
 \\
 &= \frac{1}{4\pi^2}\int_0^\infty \int_{-\infty}^\infty 
 \int_{-\infty}^\infty
 \frac{\widehat{K}((-\delta+i\omega)h)}{\widehat{\lambda}(\delta-i\omega)} 
\overline{ \frac{\widehat{K}((-\delta+i\nu)h)}{\widehat{\lambda}(\delta-i\nu)}} e^{-2\delta (t-x_0)+i(\omega-\nu)(t-x_0)} t^{p}
\rd\omega \rd \nu \rd t
\\
&=
\frac{1}{4\pi^2}
\int_{-\infty}^\infty 
 \int_{-\infty}^\infty
 \frac{\widehat{K}((-\delta+i\omega)h)}{\widehat{\lambda}(\delta-i\omega)} 
\overline{ \frac{\widehat{K}((-\delta+i\nu)h)}{\widehat{\lambda}(\delta-i\nu)}} e^{2\delta x_0 - i(\omega-\nu)x_0} 
\frac{\Gamma(p+1)}{(2\delta + i(\nu-\omega))^{p+1}}
\rd\omega \rd \nu 
\\
&\leq 
\frac{\Gamma(p+1) e^{2\delta x_0} }{4\pi^2} \int_{-\infty}^\infty 
 \int_{-\infty}^\infty 
 \bigg|\frac{\widehat{K}((-\delta+i\omega)h)}{\widehat{\lambda}(\delta-i\omega)}\bigg|^2 
 \big[4\delta^2+ (\omega-\nu)^2\big]^{-\frac{1}{2}(p+1)} \rd \omega \rd \nu 
\\
&\leq c_{6} e^{2\delta x_0} \delta^{-p}\int_{-\infty}^\infty 
 \bigg|\frac{\widehat{K}((-\delta+i\omega)h)}{\widehat{\lambda}(\delta-i\omega)}\bigg|^2\rd \omega
 \leq c_{7}\lambda_0^{-2}\delta^{-p} e^{2\delta (x_0+h)} h^{-2\gamma-1},
 \end{align*}
which together with  (\ref{eq:V-1}) yields 
\begin{equation}\label{eq:var=0}
 {\rm var}_G[\tilde{G}_h(x_0)] \leq 
 c_{8}(a_1+a_2\delta^{-p}) M e^{2\delta (x_0+h)}(\lambda_0 nh^{2\gamma+1})^{-1}.
\end{equation}
Combining (\ref{eq:var>0}) and (\ref{eq:var=0}) we conclude that for  any $\sigma_\lambda\geq 0$
one has
\begin{equation*}
 {\rm var}_G[\tilde{G}_h(x_0)] \leq c_{9} 
 \kappa_3 (\sigma_\lambda, \delta) Me^{2(\sigma_\lambda+\delta)(x_0+h)}\big(\lambda_0 n h^{2\gamma+1}\big)^{-1},
\end{equation*}
where 
\begin{equation*}
 \kappa_3(\sigma_\lambda, \delta) :=  \left\{\begin{array}{ll}
(\sigma_\lambda+\delta)^{-1}, & \sigma_\lambda >0,\\*[2mm]
a_1+a_2\delta^{-p}, & \sigma_\lambda=0.
                                          \end{array}
\right.
\end{equation*}
\par\medskip 
4$^0$. Now we are in a position to finish the proof of (\ref{eq:upper-bound}).
For this purpose we specify the choice of parameters $h$ and $\delta$ and substitute the corresponding 
expressions in the derived bounds on bias and variance.  
\par 
(a). Consider first the case $\sigma_\lambda>0$. Here we put
\[
 h=h_*= \bigg(\frac{M e^{2 \sigma_\lambda x_0}}{A^2\sigma_\lambda \lambda_0 n}\bigg)^{1/(2\beta+2\gamma+1)},
\]
and 
\[
 \delta=\delta_*:=\frac{c_{1}}{T-x_0}
 \ln \bigg\{M^{\frac{\beta+\gamma}{2\beta+2\gamma+1}} (Ae^{-\sigma_\lambda x_0})^{\frac{1}{2\beta+2\gamma+1}}
 (\sigma_\lambda \lambda_0 n)^{\frac{\beta+\gamma+1}{2\beta+2\gamma+1}}\bigg\} 
\]
with some appropriately large absolute constant $c_{1}$.
Note that $\delta_*\to 0$ as $n\to\infty$ because $\frac{1}{T}\ln (\lambda_0 n)\to 0$ as $n\to\infty$.
With the choice of $h=h_*$ and $\delta=\delta_*$,
it is straightforward to verify in \eqref{eq:bias} that the bias is bounded above by a multiple of $Ah_*^\beta$.
Then substituting the expressions for $h_*$ and $\delta_*$ 
in (\ref{eq:bias}) and (\ref{eq:var>0}) and taking the limit as $n\to \infty$ 
we obtain (\ref{eq:upper-bound}) for the case $\sigma_\lambda>0$.
\par 
(b). Now, let $\sigma_\lambda=0$ and assume first that $p=0$, i.e., the intensity function $\lambda(t)$ is bounded.
Here we choose 
\[
 h=h_*=\bigg(\frac{M}{A^2\lambda_0 n}\bigg)^{1/(2\beta+2\gamma+1)}
\]
and 
\[
 \delta=\delta_*=\frac{c_{2}}{T-x_0}
 \ln \bigg\{M^{\frac{\beta+\gamma}{2\beta+2\gamma+1}} A^{\frac{1}{2\beta+2\gamma+1}}
 (\lambda_0 n)^{\frac{\beta+\gamma+1}{2\beta+2\gamma+1}}\bigg\} 
\]
with appropriate constant $c_{2}$.
Substituting these values in (\ref{eq:bias}) and (\ref{eq:var=0}) we obtain (\ref{eq:upper-bound}).
\par
If  $p>0$ then we  set 
\[
 h=h_*:=\bigg[\frac{M (a_1+a_2x_0^p)}{A^2\lambda_0 n}\bigg]^{1/(2\beta+2\gamma+1)},
\]
and  $\delta=\delta_*:=\frac{p}{2x_0}$. It is easily checked that under this choice 
(\ref{eq:upper-bound}) holds with 
\[
 \varphi_n=A^{(2\gamma+1)/(2\beta+2\gamma+1)} 
 \bigg[\frac{M(a_1+a_2x_0^p)}{\lambda_0 n}\bigg]^{\beta/(2\beta+2\gamma+1)}.
\]
%
\subsection{Proof of Theorems~\ref{th:expectation-1} and \ref{th:expectation-2}}

We provide a unified proof of Theorems~\ref{th:expectation-1} and~\ref{th:expectation-2}. Each step of the proof considers 
two cases: $\sigma_\lambda>0$ (corresponds to the proof of Theorem~\ref{th:expectation-1})
and $\sigma_\lambda=0$ (corresponds to the proof of Theorem~\ref{th:expectation-2}).
\par
Throughout the proof $c_1, c_2,\ldots$ stand for constants depending on $\gamma$, $k_0$ and $p$ only. In particular, 
they do not depend on $M$, $\lambda_0$ and $\sigma_\lambda$.
In order to avoid additional technicalities in the sequel we assume that $\gamma$ is integer. 
\par
1$^0$. First we bound the bias of $\tilde{\mu}_G$. By Lemma~\ref{lem:2}  
\begin{eqnarray*}
 \bE_G [\tilde{\mu}_G] = \int_0^T J_b(t) H(t)\rd t = \int_{0}^\infty \psi_{b}(x) \bar{G}(x)\rd x - \int_{T}^\infty J_b(t) H(t)\rd t.
\end{eqnarray*}
Therefore 
\begin{equation}\label{eq:bias-M}
 \big|\bE_G [\tilde{\mu}_G] - \mu_G \big| \leq \int_b^\infty \bar{G}(x)\rd x + \int_{T}^\infty |J_b(t)| H(t)\rd t.
\end{equation}
Since $G\in \sG(M)$,   
\begin{equation}\label{eq:B-0}
 \int_b^\infty \bar{G} (x)\rd x \leq \int_b^\infty \frac xb \bar{G}(x)\rd x \leq  Mb^{-1}.
\end{equation}
Our current goal is to bound the second integral on the right hand side of (\ref{eq:bias-M}).
\par 
Let $s=-\sigma$, where $\sigma:=\sigma_\lambda+\delta$ with some parameter $\delta>0$ to be specified. By definition of 
$J_b(t)$ and by Assumption~\ref{as:lambda}
\begin{eqnarray*}
 |J_b(t)| &\leq& \frac{e^{-\sigma t}}{2\pi \lambda_0 k_0} \int_{-\infty}^\infty |\widehat{\psi}_b(-\sigma+i\omega)| 
 \big[(\sigma-\sigma_\lambda)^2+\omega^2\big]^{\gamma/2}\rd \omega 
 \\
 &\leq & 
 c_1 e^{-\sigma t}\lambda_0^{-1} \bigg\{ 
 \delta^\gamma\int_{-\infty}^\infty |\widehat{\psi}_b(-\sigma+i\omega)| \rd \omega + 
 \int_{-\infty}^\infty |\widehat{\psi}_b(-\sigma+i\omega)|\cdot |\omega|^\gamma \rd \omega\bigg\}.
\end{eqnarray*}
Repeatedly integrating by parts the integral $\int_{-1/4}^{b+1/4} \psi_b(t) e^{(\sigma-i\omega)t}\rd t$ we obtain for 
any $r=1,2,\ldots$
that
\[
|\widehat{\psi}_b(-\sigma +i\omega)| \leq c_2(r) e^{\sigma (b+1/4)}(\sigma^2+\omega^2)^{-r/2}, 
\]
where constant $c_2(r)$ depends on $r$.
This inequality with $r>1$  yields
\begin{eqnarray*}
 \int_{-\infty}^\infty |\widehat{\psi}_b(-\sigma +i\omega)|\rd \omega &\leq& c_3(r) e^{\sigma (b+1/4)}\sigma^{-r+1}
\end{eqnarray*}
 and with $r>\gamma+1$
 \begin{eqnarray*}
\int_{-\infty}^\infty 
 |\widehat{\psi}_b(-\sigma +i\omega)| \cdot |\omega|^\gamma \rd \omega \leq c_4(r)  e^{\sigma (b+1/4)} \sigma^{-r+\gamma+1}.
 \end{eqnarray*}
Therefore the following bound holds
\[
 |J_b(t)|\leq c_5(r)\lambda_0^{-1} e^{-\sigma (t-b-1/4)} \Big\{\delta^\gamma \sigma^{-r+1} + \sigma^{-r+\gamma+1}\Big\},
\]
where $r>\gamma+1$.
\par
(a). First, consider the case
$\sigma_\lambda>0$. By Assumption~\ref{as:lambda-1}(a),
$H(t)\leq M\lambda_0 e^{\sigma_\lambda t}$; hence  
\begin{equation}\label{eq:B-1}
 \int_{T}^\infty |J_b(t)| H(t)\rd t \leq \frac{c_5(r)M}{\sigma^{r-1}\delta} 
 (\delta^\gamma + \sigma^{\gamma}) e^{\sigma (b+1/4)} e^{-\delta T}.
\end{equation}
Combining  
(\ref{eq:B-1}), (\ref{eq:B-0}) and (\ref{eq:bias-M})
we obtain
\begin{equation*}
 \big|\bE_G [\tilde{\mu}_G] - \mu_G \big| \leq Mb^{-1} + c_5(r) M
 \delta^{-1}(\sigma_\lambda+\delta)^{-r+\gamma+1} e^{(\sigma_\lambda+\delta) (b+1/4)} e^{-\delta T}.
\end{equation*}
In particular, if we set $r=\gamma+2$ and if 
\begin{equation}\label{eq:T-large-1}
T-b-\tfrac{1}{4}  \geq \frac{1}{\delta} \ln \bigg(\frac{b e^{\sigma_\lambda (b+1/4)}}{\sigma_\lambda}\bigg)
\end{equation}
then 
\begin{equation}\label{eq:bias->0}
 \big|\bE_G [\hat{\mu}_G] - \mu_G \big| \leq c_6 Mb^{-1}.
\end{equation}
\par
(b). Now let  $\sigma_{\lambda}=0$. In view of  Assumption~\ref{as:lambda-1}(b), $H(t) \leq M\lambda_0(a_1+a_2t^p)$
so that 
\begin{eqnarray*}
 \int_{T}^\infty |J_b(t)| H(t)\rd t \leq
c_5(r) \delta^{\gamma-r+1} M e^{\delta(b+1/4)} \bigg[a_1 \delta^{-1} e^{-\delta T} + 
a_2\int_T^\infty t^p e^{-\delta t}\rd t\bigg].
 \end{eqnarray*}
If $\delta T\geq 2p$ then using Lemma~\ref{lem:Gamma-function} we obtain
\begin{eqnarray}\label{eq:B-2}
 \int_{T}^\infty |J_b(t)| H(t)\rd t \leq
c_7(r) \delta^{\gamma-r} M e^{-\delta(T-b-1/4)} \big(a_1  +  a_2\delta ^{-1} T^{p-1}\big).
 \end{eqnarray}
It follows from (\ref{eq:B-2}) that if 
\begin{equation}\label{eq:T-large-2}
 T-b-\frac{1}{4} \geq \frac{1}{\delta}\ln \Big\{b \delta^{\gamma-r}(a_1+a_2\delta^{-1} T^{p-1})\Big\},\;\;T\geq \frac{2p}{\delta}
\end{equation}
then again (\ref{eq:bias->0}) holds. 
\par 
We conclude that the bound (\ref{eq:bias->0}) on the bias holds, 
provided that $b$ and $\delta$ are chosen to satisfy (\ref{eq:T-large-1}) in case $\sigma_\lambda>0$ and 
(\ref{eq:T-large-2}) in case $\sigma_\lambda=0$.
\par\medskip  
2$^0$. Now we proceed with bounding the variance of $\tilde{\mu}_G$. We have 
\begin{eqnarray}\label{eq:VAR}
 {\rm var}_G [\tilde{\mu}_G] = \frac{1}{n} \int_0^T \int_0^T J_b(t) \overline{J_b(\tau)} R(t,\tau)\rd t \rd \tau
 \leq \frac{1}{n} \int_0^\infty |J_b(t)|^2\int_0^\infty R(t,\tau)\rd \tau \rd t,
\end{eqnarray}
where $R(t,\tau)={\rm cov}_G [X(t), X(\tau)]$ is given in (\ref{eq:R}). 
\par 
(a). Let $\sigma_\lambda>0$; then, 
applying  Lemma~\ref{lem:covariance}(a)
with $\nu=\sigma_\lambda +2\delta$ we have 
\begin{eqnarray}\label{eq:var-mu}
 {\rm var}_G[\tilde{\mu}_G] &\leq&  \frac{M\lambda_0}{n(\sigma_\lambda +\delta)} 
 \int_0^\infty |J_b(t)|^2 e^{2(\sigma_\lambda+\delta)}\rd t
\nonumber
 \\
 &\leq&  \frac{M\lambda_0}{4\pi^2n(\sigma_\lambda +\delta)}\int_{-\infty}^\infty \bigg|\frac{\widehat{\psi}_b(-\sigma_\lambda-\delta+i\omega)}
 {\widehat{\lambda}(\sigma_\lambda+\delta -i\omega)}\bigg|^2 \rd \omega.
\end{eqnarray}
We recall 
that, by Assumption~\ref{as:lambda}, for $\sigma=\sigma_\lambda+\delta$, $\delta>0$ one has
\begin{eqnarray*}
 \int_{-\infty}^\infty \bigg|\frac{\widehat{\psi}_b(-\sigma +i\omega)}
 {\widehat{\lambda}(\sigma
 -i\omega)}
 \bigg|^2\rd \omega \leq \frac{1}{\lambda_0^2 k_0^2}  \int_{-\infty}^\infty |\widehat{\psi}_b(-\sigma+i\omega)|^2
 [\delta^2 + \omega^2]^\gamma \rd \omega,
\end{eqnarray*}
and the integral on the right hand side is bounded as follows.  
First, by Parseval's identity,
\begin{eqnarray}\label{eq:I-1}
 \int_{-\infty}^\infty |\widehat{\psi}_b(-\sigma +i\omega)|^2\rd \omega 
 =  2\pi \int_{-1/4}^{b+1/4}
 e^{2\sigma t}\psi_b^2(t) \rd t \leq \pi e^{2\sigma(b+1/4)}\sigma^{-1}.
 \end{eqnarray}
 Second,
 for integer $\gamma\geq 0$  we have 
 \begin{multline*}
 \int_{-\infty}^\infty |\widehat{\psi}_b(-\sigma +i\omega)|^2 |\omega|^{2\gamma}\rd \omega
 = 2\pi \int_{-1/4}^{b+1/4} \bigg|\frac{\rd^\gamma }{\rd t^\gamma}[e^{\sigma t}\psi_b(t)]\bigg|^2 \rd t
 \\
 = 2\pi \int_{-1/4}^{b+1/4} \bigg| \sum_{j=0}^\gamma \tbinom{\gamma}{j} \sigma^j e^{\sigma t} \psi^{(\gamma-j)}_b(t)\bigg|^2\rd t
 \leq c_1 \sum_{j=0}^\gamma \sigma^{2j} \int_{-1/4}^{b+1/4} e^{2\sigma t} |\psi^{(\gamma-j)}_b(t)|^2\rd t.
\end{multline*}
By construction, derivatives of $\psi_b$ are 
nonzero only on the intervals $[-1/4,0]$ and $[b,b+1/4]$; therefore 
\begin{equation}\label{eq:I-2}
\int_{-\infty}^\infty |\widehat{\psi}_b(-\sigma +i\omega)|^2 |\omega|^{2\gamma}\rd \omega \leq
 c_2 e^{2\sigma(b+1/4)}\Big[\sigma^{2\gamma}(b+\tfrac{1}{2}) + \sum_{j=0}^{\gamma-1} \sigma^{2j}\Big].
\end{equation}
Combining (\ref{eq:I-1}) and (\ref{eq:I-2}) we obtain 
\[
 \int_{-\infty}^\infty \bigg|\frac{\widehat{\psi}_b(-\sigma +i\omega)}
 {\widehat{\lambda}(\sigma
 -i\omega)}
 \bigg|^2\rd \omega \leq c_3\lambda_0^{-2} e^{2\sigma(b+1/4)}
 \bigg(\delta^{2\gamma}\sigma^{-1} + \sigma^{2\gamma}(b+\tfrac{1}{2}) + \sum_{j=0}^{\gamma-1}\sigma^{2j}\bigg)
\]
which along with (\ref{eq:var-mu}) yields
\begin{eqnarray*}
 {\rm var}_G[\tilde{\mu}_G] \leq \frac{c_4 M e^{2(\sigma_\lambda+\delta)(b+1/4)}}{\lambda_0 n(\sigma_\lambda +\delta)}
 \bigg[\frac{\delta^{2\gamma}}{\sigma_\lambda+\delta}
 + (\sigma_\lambda+\delta)^{2\gamma} (b+\tfrac{1}{2}) + 
 \sum_{j=0}^{\gamma-1} (\sigma_\lambda+\delta)^{2j}\bigg].
\end{eqnarray*}
\par
We set  
\[
\delta=\delta_*:= \frac{1}{b};
\] 
also $b$ will be always chosen so that $b\geq \sigma_\lambda^{-1}$ for large $n$.
Under these conditions we finally obtain
\[
 {\rm var}_G [\tilde{\mu}_G] \leq c_{5} M \sigma_\lambda^{2\gamma-1} be^{2\sigma_\lambda(b+1/4)} (\lambda_0 n)^{-1}.
\]
We note also that with the choice $\delta=\delta_*$ the bound on the bias (\ref{eq:bias->0}) holds provided that 
\[ 
T\geq b +b\sigma_\lambda (b+1/4)+ b\ln (b/\sigma_\lambda);
\]
see (\ref{eq:T-large-1}).
\par 
(b). Now consider the case $\sigma_\lambda=0$.  
It follows from (\ref{eq:VAR}) and Lemma~\ref{lem:covariance} that 
\[
 {\rm var}_G[\tilde{\mu}_G] \leq \frac{c_6M\lambda_0}{n} \int_0^T |J_b(t)|^2 (a_1 + a_2 t^{p})\rd t.
\]
Since $|\widehat{\lambda}(\sigma+i\omega)|^{-1}$ is 
finite for all $\sigma\geq 0$ and $\omega\in \bR$, function $\widehat{\psi}_b(z)/\widehat{\lambda}(-z)$ 
can be analytically continued to the imaginary axis. Therefore,
by 
Parseval's identity and Assumption~\ref{as:lambda} we obtain
\begin{align*}
 \int_{-\infty}^\infty |J_b(t)|^2 \rd t &= \frac{1}{4\pi^2}\int_{-\infty}^\infty \bigg|
 \frac{\widehat{\psi}_b(i\omega)}{\widehat{\lambda} (-i\omega)}\bigg|^2\rd \omega
 \leq \frac{1}{4\pi^2\lambda_0^2k_0^2} \int_{-\infty}^\infty |\widehat{\psi}_b(i\omega)|^2 |\omega|^{2\gamma}\rd \omega
 \\
&\leq c_{7} \lambda_0^{-2}
 \int_{-1/4}^{b+1/4} |\psi_b^{(\gamma)}(t)|^2 \rd t
\leq c_{8}\lambda_0^{-2},
 \end{align*}
and  for $\delta>0$
\begin{align*}
 \int_0^\infty |J_b(t)|^2 t^{p+1} \rd t &= \frac{1}{4\pi^2}\int_0^\infty \int_{-\infty}^\infty \int_{-\infty}^\infty 
 \frac{\widehat{\psi}_b(-\delta+i\omega)}{\widehat{\lambda}(\delta-i\omega)} 
\overline{ \frac{\widehat{\psi}_b(-\delta+i\nu)}{\widehat{\lambda}(\delta-i\nu)}} e^{-2\delta t+i(\omega-\nu)t} t^{p}
\rd\omega \rd \nu \rd t
\\
&=
\frac{1}{4\pi^2}
\int_{-\infty}^\infty 
 \int_{-\infty}^\infty
 \frac{\widehat{\psi}_b(-\delta+i\omega)}{\widehat{\lambda}(\delta-i\omega)} 
\overline{ \frac{\widehat{\psi}_b(-\delta+i\nu)}{\widehat{\lambda}(\delta-i\nu)}}  
\frac{\Gamma(p+1)}{(2\delta + i(\nu-\omega))^{p+1}}
\rd\omega \rd \nu 
\\
&\leq 
\frac{\Gamma(p+1) }{4\pi^2} \int_{-\infty}^\infty 
 \int_{-\infty}^\infty 
 \bigg|\frac{\widehat{\psi}_b(-\delta+i\omega)}{\widehat{\lambda}(\delta-i\omega)}\bigg|^2 
 \big[4\delta^2+ (\omega-\nu)^2\big]^{-\frac{1}{2}(p+1)} \rd \omega \rd \nu 
\\
&\leq \frac{c_{9}}{\delta^{p}}\int_{-\infty}^\infty 
 \bigg|\frac{\widehat{\psi}_b(-\delta+i\omega)}{\widehat{\lambda}(\delta-i\omega)}\bigg|^2\rd \omega.
\end{align*}
Therefore
\[
 {\rm var}_G[\tilde{\mu}_G] \leq c_{10} \frac{M\lambda_0}{n} (a_1+ a_2\delta^{-p})
 \int_{-\infty}^\infty 
 \bigg|\frac{\widehat{\psi}_b(-\delta+i\omega)}{\widehat{\lambda}(\delta-i\omega)}\bigg|^2\rd \omega.
\]
Using the same reasoning as before
\[
 \int_{-\infty}^\infty 
 \bigg|\frac{\widehat{\psi}_b(-\delta+i\omega)}{\widehat{\lambda}(\delta-i\omega)}\bigg|^2\rd \omega
\leq c_{11}\lambda_0^{-2} e^{2\delta(b+1/4)}\Big[\delta^{2\gamma}(b+\tfrac{1}{2}) + \sum_{j=0}^{\gamma-1} \delta^{2j} + \delta^{2\gamma-1} \Big].
 \]
We again set $\delta=\delta_*:=b^{-1}$; with this choice
\[
 {\rm var}_G[\tilde{\mu}_G] \leq c_{12} M(a_1+ a_2 b^{p})(\lambda_0 n)^{-1},
\]
and the bound on the bias is given in (\ref{eq:bias->0}) provided that
\[
 T\geq  b+ \frac{1}{4} + b\ln b + b\ln (a_1+a_2b T^{p-1}), \;\; T\geq 2p\ln b;
\]
see (\ref{eq:T-large-2}).
\par 
3$^0$. Now we complete the proof by selecting $b$ to balance the obtained bounds on the bias and variance. 
\par 
(a). Consider the case $\sigma_\lambda>0$. Let
\[
 b=b_*=\frac{1}{2\sigma_\lambda} \bigg\{\ln \bigg(\frac{M\lambda_0  n}{\sigma_\lambda^{2\gamma-1}}\bigg)- 3
 \ln \bigg[\frac{1}{2\sigma_\lambda}\ln \bigg(\frac{M\lambda_0n}{\sigma_\lambda^{2\gamma-1}}\bigg)\bigg]\bigg\}-\frac{1}{4}.
\]
We note that because $[\ln (\lambda_0n)]^2/T\to 0$ as $n\to\infty$, 
(\ref{eq:T-large-1}) is fulfilled for large $n$ with $b=b_*$ and $\delta=\delta_*=1/b_*$.
Therefore substituting the expression for $b_*$ in the bounds for the bias and variance we obtain  
\[
 \limsup_{n\to \infty} \Big\{\varphi_n^{-1} \cR[\tilde{\mu}_G; \sG(M)\big] \Big\}\leq C_1,\;\;\;
 \varphi_n := M \sigma_\lambda  \ln \bigg(\frac{M\lambda_0 n}{\sigma_\lambda^{2\gamma-1}}\bigg).
\]
\par 
(b). Now, let $\sigma_\lambda=0$, assume that $a_2=0$. In this case the bound on the variance is 
${\rm var}_G[\tilde{\mu}_G] \leq c_{1} M (\lambda_0 n)^{-1}$ and on the bias $Mb^{-1}$. If we choose $b\geq 
b_*\geq c_2\sqrt{M\lambda_0n}$
with appropriate absolute constant $c_2$ then under condition $\sqrt{\lambda_0 n}\ln (\lambda_0n)/T\to 0$ as $n\to\infty$ we obtain that 
\[
 \limsup_{n\to\infty} \bigg\{\sqrt{\frac{\lambda_0 n}{M}}\, \cR[\tilde{\mu}_G; \sG(M)] \bigg\} \leq C_2. 
\]
\par 
If $a_2 > 0$ then we set 
$b=b_*= \big[M\lambda_0 n\big]^{1/(p+2)}$. If 
$\frac{1}{T}(\lambda_0 n)^{1/(p+2)} \ln (\lambda_0 n)\to 0$ as $n\to \infty$
then (\ref{eq:T-large-2}) is fulfilled for large $n$
and 
\[
 \varphi_n:=M^{(p+1)/(p+2)} (\lambda_0n)^{-1/(p+2)}.
\]

\bibliographystyle{plain}
\bibliography{biblio}

\begin{thebibliography}{10}

\bibitem{Abramovich}
{\sc Abramovich, F., Pensky, M., and Rozenholc, Y.}
\newblock Laplace deconvolution with noisy observations.
\newblock {\em Electron. J. Stat. 7\/} (2013), 1094--1128.

\bibitem{Antunes}
{\sc Antunes, N., Pipiras, V., and Veitch, D.}
\newblock Skampling for the flow duration distribution.
\newblock {\em Proceedings of 29th International Teletraffic Congress (ITC 29)
  1\/} (2017), 63--71.

\bibitem{Baccelli}
{\sc Baccelli, F., Kauffman, B., and Veitch, D.}
\newblock Inverse problems in queueing theory and internet probing.
\newblock {\em Queueing Syst. 63\/} (2009), 59--107.

\bibitem{Bigot}
{\sc Bigot, J., Gadat, S., Klein, T., and Marteau, C.}
\newblock Intensity estimation of non-homogeneous {Poisson} processes from
  shifted trajectories.
\newblock {\em Electron. J. Stat. 7\/} (2013), 881--931.

\bibitem{Bingham}
{\sc Bingham, N.~H., and Pitts, S.~M.}
\newblock Non--parametric estimation for the {$M/G/\infty$} queue.
\newblock {\em Ann. Inst. Statist. Math. 51\/} (1999), 71--97.

\bibitem{Borwein09}
{\sc Borwein, J.~M., and O-Yeat, C.}
\newblock Uniform bounds for the complementary incomplete gamma function.
\newblock {\em Math. Inequal. Appl. 12\/} (2009), 115 -- 121.

\bibitem{Brown}
{\sc Brown, M.}
\newblock An {$M/G/\infty$} estimation problem.
\newblock {\em Ann. Math. Statist. 41\/} (1970), 651 -- 654.

\bibitem{Coifman}
{\sc Coifman, B., and Cassidy, M.}
\newblock Vehicle reidentification and travel time measurement on congested
  freeways.
\newblock {\em Transportation Research Part A: Policy and Practice 36}, 10
  (2002), 899 -- 917.

\bibitem{Dob2009}
{\sc Dobrzy{\'n}ski, M., and Bruggeman, F.~J.}
\newblock Elongation dynamics shape bursty transcription and translation.
\newblock {\em Proceedings of the National Academy of Sciences 106}, 8 (2009),
  2583--2588.

\bibitem{Doetsch}
{\sc Doetsch, G.}
\newblock {\em Introduction to the {T}heory and {A}pplication of the {L}aplace
  {T}ransformation}.
\newblock Springer-Verlag, Berlin, 1974.

\bibitem{Durbin}
{\sc Durbin, F.}
\newblock Numerical inversion of {Laplace} transforms: an efficient improvement
  to {Dubner} and {Abate's} method.
\newblock {\em Comput. J. 17}, 4 (1973), 371--376.

\bibitem{EMW1993}
{\sc Eick, S., Massey, W., and Whitt, W.}
\newblock The physics of the {$M_t/G/\infty$} queue.
\newblock {\em Management Science 39}, 2 (1993), 241--252.

\bibitem{fan}
{\sc Fan, J.}
\newblock On the optimal rates of convergence for nonparametric deconvolution
  problems.
\newblock {\em Ann. Statist. 19}, 3 (1991), 1257--1272.

\bibitem{Goldenshluger}
{\sc Goldenshluger, A.}
\newblock Nonparametric estimation of the service time distribution in the
  {$M/G/\infty$} queue.
\newblock {\em Adv. in Appl. Probab. 48}, 4 (2016), 1117--1138.

\bibitem{Goldenshluger-b}
{\sc Goldenshluger, A.}
\newblock The {$M/G/\infty$} estimation problem revisited.
\newblock {\em Bernoulli 24\/} (2018), 2531--2568.

\bibitem{GN}
{\sc Goldenshluger, A., and Nemirovski, A.}
\newblock On spatially adaptive estimation of nonparametric regression.
\newblock {\em Math. Methods Statist. 6}, 2 (1997), 135--170.

\bibitem{Hille}
{\sc Hille, E.}
\newblock {\em Analytic Function Theory}, vol.~II.
\newblock American Mathematical Society Chelsey Publishing, Providence, RI.,
  1962.

\bibitem{Jansen2016}
{\sc Jansen, H.~M., Mandjes, M. R.~H., De~Turck, K., and Wittevrongel, S.}
\newblock A large deviations principle for infinite-server queues in a random
  environment.
\newblock {\em Queueing Systems 82}, 1 (Feb 2016), 199--235.

\bibitem{Kris1997}
{\sc Krishnamurthy, V., and Elliot, R.}
\newblock Filters for estimating {Markov} modulated {Poisson} processes and
  image-based tracking.
\newblock {\em Automatica 33}, 5 (1997), 821 -- 833.

\bibitem{kutoyants}
{\sc Kutoyants, Y.~A.}
\newblock {\em Statistical {I}nference for {S}patial {P}oisson {P}rocesses},
  vol.~134 of {\em Lecture Notes in Statistics}.
\newblock Springer-Verlag, New York, 1998.

\bibitem{Lepski}
{\sc Lepski\u\i, O.~V.}
\newblock A problem of adaptive estimation in {G}aussian white noise.
\newblock {\em Teor. Veroyatnost. i Primenen. 35}, 3 (1990), 459--470.

\bibitem{Wichelhaus}
{\sc Schweer, S., and Wichelhaus, C.}
\newblock Nonparametric estimation of the service time distribution in the
  discrete-time {$GI/G/\infty$} queue with partial information.
\newblock {\em Stoch. Process. Appl. 125\/} (2015), 233 -- 253.

\bibitem{Tsybakov}
{\sc Tsybakov, A.~B.}
\newblock {\em Introduction to {N}onparametric {E}stimation}.
\newblock Springer Series in Statistics. Springer, New York, 2009.

\bibitem{Widder46}
{\sc Widder, D.}
\newblock {\em The {Laplace} Transform}.
\newblock Princeton University Press, Princeton, 1946.

\bibitem{Zhang}
{\sc Zhang, C.-H.}
\newblock Fourier methods for estimating mixing densities and distributions.
\newblock {\em Ann. Statist. 18}, 2 (1990), 806--831.

\end{thebibliography}

\end{document}